\documentclass{amsart}
\usepackage{amssymb,enumerate, bbm}
\usepackage[chatter]{rotating}
\usepackage{epsfig}
\usepackage{fancyvrb}
\newtheorem{lemma}{Lemma}[section]

\newtheorem{theorem}[lemma]{Theorem}
\newtheorem{remark}[lemma]{Remark}
\newtheorem{cor}[lemma]{Corollary}
\newtheorem{proposition}[lemma]{Proposition}

\newcommand{\N}{{\mathbb N}}

\newcommand{\A}{\mathcal A}

\newcommand{\Ro}{{\overline R}}

\newcommand{\Vo}{{\overline V}}

\newcommand{\ad}{{\rm ad}}
\newcommand{\aV}{{\hat V}}

\newcommand{\aR}{{\hat R}}

\newcommand{\ba}{\mathbf { a}}
\newcommand{\bc}{\mathbf { c}}
\newcommand{\bw}{\mathbf { w}}
\newcommand{\bu}{\mathbf { u}}
\newcommand{\bv}{\mathbf { v}}
\newcommand{\bs}{\mathbf { s}}

\newcommand{\Q}{{\mathbb Q}}

\newcommand{\F}{{\mathbb F}}

\newcommand{\Z}{{\mathbb Z}}

\newcommand{\lm}{\lambda}
\newcommand{\lmu}{\lambda_1}
\newcommand{\lmf}{\lambda_1^f}
\newcommand{\lmd}{\lambda_2}
\newcommand{\lmdf}{\lambda_2^f}

\newcommand{\al}{\alpha}
\newcommand{\bt}{\beta}
\newcommand{\M}{\mathcal M}

\newcommand{\Mab}{\mathcal M(\alpha, \beta)}

\usepackage{xcolor}

\title[]{$2$-generated axial algebras of Monster type $(2\bt, \bt)$}
  
\author{Clara Franchi}
\address{Dipartimento di Matematica e Fisica,
Universit\`a Cattolica del Sacro Cuore,
Via della Garzetta 48,
I-25133 Brescia, Italy}
\email{clara.franchi@unicatt.it}

\author{ Sergey Shpectorov}
\address{School of Mathematics,
University of Birmingham, 
Watson Building, Edgbaston,
Birmingham, B15 2TT, UK}
\email{s.shpectorov@bham.ac.uk}

\author{Mario Mainardis}
\address{Dipartimento di Scienze Matematiche, Informatiche e Fisiche, 
Universit\`a degli Studi di Udine, via delle Scienze 206,
I-33100 Udine, Italy}
\email{mario.mainardis@uniud.it}

\begin{document}

\begin{abstract}
Axial algebras of Monster type $(\al, \bt)$ are a class of non-associative algebras that includes, besides associative algebras, other important examples such as the Jordan algebras and the Griess algebra. $2$-generated  primitive  axial algebras of Monster type $(\al, \bt)$ naturally split into three cases: the case when $\al\not \in \{2\bt, 4\bt\}$, the case $\al=4\bt$ and $\al=2\bt$. In this paper we give a complete classification  all $2$-generated  primitive  axial algebras of Monster type $(2\beta, \beta)$.\end{abstract}

\maketitle

\section{Introduction}

This paper is part of a programme aimed at classifying all $2$-generated 
 axial algebras of Monster type $(\alpha, \beta)$ over a field of characteristic other than $2$. Such algebras appear in different areas of mathematics and physics and are of particular interest for the study of several classes of finite simple groups,  such  as the  $3$-transposition groups, and many of the sporadics, including the Monster (see~\cite{FMI, IvInd}  and the introductions of~\cite{R, KMS}, and ~\cite{FMS1}).  Just like for Lie algebras, a fundamental step towards the understanding of axial algebras of Monster type is the classification of the $2$-generated ones. This project  has its origins in a work by S. Norton~(see~\cite{N96}), who classified the $2$-generated subalgebras of the Griess algebra (which is a real axial algebra of Monster type $(\tfrac{1}{4},\tfrac{1}{32})$).   The general case was first attacked by F.~Rehren ~\cite{R, RT}, who  proved that every $2$-generated primitive axial algebra of Monster type $(\alpha, \beta)$ over a ring $R$ in which $2$, $\alpha$, $\bt$, $\alpha-\bt$,  $\alpha-2\bt$, and $\alpha-4\bt$ are invertible can be generated as $R$-module by $8$ vectors and computed the structure constants with respect to these elements.  In that paper the main case subdivision for this project is implicitly indicated, namely:
 \begin{enumerate}
 \item[-] {\it the $\alpha\not \in\{2\beta, 4\beta\}$ case}; 
 \item[-]{\it the $\alpha=2\beta$ case};
 \item[-]{\it the $\alpha=4\beta$ case}.
 \end{enumerate}
Each of these cases is divided into two further subcases, depending whether the algebra admits an automorphism that swaps the two generating axes (the {\it symmetric case}) or not  (the {\it non symmetric case}).
 
 At present, results of T. Yabe~\cite{Yabe}, C. Franchi, M.Mainardis, S. Shpectorov, and J. McInroy~\cite{HW, FM, FMM},  give a complete understanding of the symmetric case.
 
 The non symmetric case is still wide open. A first result in this direction is the classification of the $2$-generated primitive axial algebras over $\Q(\alpha,\beta)$, where $\alpha$ and $\beta$ are algebraically independent indeterminates over $\Q$ obtained in~\cite{FMS1}. In that case it turns out that every such algebra is symmetric. There are, however, non symmetric examples, e.g. the algebra $Q_2(\beta)$ constructed in~\cite{JoshiPhD, DM}, by V.~Joshi.   Another family of examples, pointed out by Michael Turner~\cite{Turner}, is given by the Matsuo algebras  
$3C(\al)$. Such an algebra has basis $b_0,b_1,b_2$ and multiplication defined by $ab=\frac{\al}{2}(a+b-c)$, whenever $\{a,b,c\}=\{b_0, b_1, b_2\}$; each $b_i$ is an axis of Jordan type $\al$ and in fact the algebra itself is a primitive $2$-generated algebra of Jordan type $\al$. Provided $\al\neq -1$, it has an identity $\mathbbm 1=\frac{1}{1+\al}(b_0+b_1+b_2)$. It is straighforward to see that, for $i\in \{1,2,3\}$,  $a_i:=\mathbbm 1-b_i$ is an axis of Jordan type $1-\al$. Hence, for $\al\neq \frac{1}{2}$, with respect to the generators $a_i, b_i$ (for any fixed $i$), the algebra is a primitive $2$-generated algebra of Monster type $(\al, 1-\al)$ and it is clearly non-symmetric. We'll denote this algebra by $3C(\al, 1-\al)$.
 
 In this paper we deal with the case $\alpha=2\beta$ and give a complete classification  of such algebras, namely we prove

 \begin{theorem}\label{thm2} 
  Let $V$ be a  primitive axial algebra of Monster type $(2\bt, \bt)$ over a field $\F$ of characteristic other than $2$. Then one of the following holds
\begin{enumerate}
\item $V$ is symmetric; 
\item $V$ is isomorphic to $Q_2(\bt)$, or, when $\bt=-\frac{1}{2}$, to the $3$-dimensional quotient of $Q_2(\bt)$;
\item $V$ is isomorphic to the algebra $3C(\frac{2}{3}, \frac{1}{3})$.  
\end{enumerate}
\end{theorem}
 
The paper is organised as follows: in Section~\ref{table} we recall the basic definitions and properties of axial algebras which will be needed in the sequel. In Section~\ref{mult} we construct a universal object for the category of $2$-generated primitive axial algebras of Monster type $(2\beta, \beta)$, in a similar way as we did  in~\cite[Section 4]{FMS1} for the case where $\al\not \in \{2\bt, 4\bt\}$, and we prove that it is linearly spanned by $8$ vectors. 
In Section~\ref{strategy} we apply the machinery developed in~\cite[Section 5]{FMS1} to the case $\alpha=2\beta$. Precisely, we consider a ring  $R$ of  characteristic other than $2$, we denote by $R_0$ the prime subring of $R$ and let $R_0[\frac{1}{2}, \bt, \frac{1}{ \beta}][x, y, z, t]$ be the polynomial ring in $4$ variables over $R_0[\frac{1}{2}, \bt, \frac{1}{ \beta}]$.  For a subset $T$ of $R_0[\frac{1}{2}, \bt, \frac{1}{ \beta}][x, y, z, t]$, let ${\mathcal V}(T)$ be the set of common zeroes of the elements of $T$. Finally we 
denote by ${\mathcal M_2}(2, R)$ a set of representatives of the isomorphism classes of  $2$-generated primitive axial algebras of Monster type $(2\beta, \beta)$ over $R$. 
With such notation, we prove the following result.

\begin{theorem}\label{nec}
Assume $R$ is a ring of characteristic other than $2$. Then there exists a subset $T \subseteq R_0(\alpha, \beta)[x, y, z, t]$ of size $4$ and a map $$\xi\colon {\mathcal M_2}(2, R)\to {\mathcal V}(T)$$ 
such that, for every $P$ in the image of $\xi$, there exists an element $V_P$ of  ${\mathcal M_2}(2, R)$ with the property that every element in $\xi^{-1}(P)$ is a quotient of $V_P$. 
\end{theorem} 

In Section~\ref{symmetric} we show how Theorem~\ref{nec} can be used to obtain  an alternative and independent proof of Yabe's classification of the primitive symmetric axial algebras of Monster type $(2\beta, \beta)$ over a field (see Theorem~\ref{thm2s}).  Having at hand the classification in the symmetric case, in Section~\ref{ns}  we deal with the  non-symmetric case and prove Theorem~\ref{thm2}.
The key observation is that, for  a  primitive $2$-generated axial algebra of Monster type $V$ with generating axes $a_0$ and $a_1$,  the orbit under the Miyamoto group of $a_i$ ($i\in \{1,2\}$) generates a primitive $2$-generated axial subalgebra of Monster type which is symmetric.


\section{Basics}\label{table}

We start by recalling the definition and basic features of axial algebras.
Let  $R$ be a ring with identity   
and let $\mathcal S$ be a finite subset of $R$ with $1\in\mathcal S$. 
A {\it fusion law} on $\mathcal S$ is a map 
$$\star\colon \mathcal S \times \mathcal S \to 2^{\mathcal S}.$$
An {\it axial algebra}  over $R$ with {\it spectrum} $\mathcal S$ and fusion law $\star$ is a 
commutative non-associative $R$-algebra $V$ generated by a set $\mathcal A$ of nonzero 
idempotents (called {\it axes}) such that, for each $a\in {\mathcal A}$, 
\begin{enumerate}
\item[(Ax1)] $ad(a):v\mapsto av$ is a semisimple endomorphism of $V$ with spectrum contained in 
$\mathcal S$;
\item[(Ax2)] for every $\lm,\mu\in\mathcal S$,  the product of a $\lm$-eigenvector and a 
$\mu$-eigenvector of $\ad_a$ is the sum of $\delta$-eigenvectors, for $\delta\in\lm\star\mu$.
\end{enumerate} 
Furthermore, $V$ is called {\it primitive} if 
\begin{enumerate}
\item[(Ax3)] $V_1=\langle a \rangle$.
\end{enumerate}
An axial algebra over $R$ is said to be {\it of Monster type} $(\al,\bt)$ if it satisfies 
the fusion law $\Mab$ given in Table~\ref{Ising}, with $\al,\bt\in R\setminus\{0,1\}$, and 
$\al\neq\bt$.
 \begin{table}
$$ \begin{array}{|c||c|c|c|c|}
\hline
   \star& 1    & 0    & \alpha      &  \beta\\
   \hline
   \hline
   1     & 1    & \emptyset    &  \alpha      & \beta\\
   \hline
   0     & \emptyset    & 0    & \alpha      & \beta\\
   \hline
    \alpha  & \alpha &  \alpha & 1,0  &  \beta\\
   \hline
   \beta&  \beta& \beta& \beta& 1,0, \alpha\\
    \hline
    \end{array}
    $$
     \caption{ Fusion law $\M( \alpha, \beta)$}\label{Ising}
    \end{table}

Let $V$ be an axial algebra of Monster type $(\al, \bt)$ and  let $a\in \mathcal A$.
   Let ${\mathcal S}^+:=\{1,0, \alpha\}$ and ${\mathcal S}^-:=\{ \beta\}$. The partition $\{{\mathcal S}^+, {\mathcal S^-}\}$ of $\mathcal S$ induces a $\Z_2$-grading on ${\mathcal S}$ which, on turn, induces, a $\Z_2$-grading $\{V_+^a, V_-^a\}$ on $V$, where $V_+^a:=V_1^a+V_0^a+V_{\alpha}^a$ and $V_-^a=V_{ \beta}^a$. It follows that, if $\tau_a$ is the map from $R\cup V$ to $R\cup V$ such that $\tau_{a|_V}$ is the multiplication by $\varepsilon$ on $V_{\epsilon}^a$ for $\varepsilon \in \{\pm1\}$ 
   and $\tau_{a|_R}$ is the identity, then   $\tau_a$  is  an involutory automorphism of $V$ (see~\cite[Proposition 3.4]{HRS}). The map  $\tau_a$ is called the {\it Miyamoto involution} associated to the axis $a$. 
By definition of $\tau_a$, the element $av-{ \beta} v$ of $V$ is $\tau_a$-invariant and, since $a$ lies in $V_+^a\leq C_V(\tau_a)$, also $av-{ \beta}(a+v)$ is $\tau_a$-invariant. In particular, by symmetry,
  
\begin{lemma}\label{invariant}
Let $a$ and $b$ be axes of $V$. Then $ab- \beta(a+b)$ is fixed by the 2-generated group $\langle \tau_a, \tau_b\rangle$.
\end{lemma}

If $V$ is generated by the set of axes $\mathcal A:=\{a_0, a_1\}$,  for $i\in \{1,2\}$. Set $\rho:=\tau_{a_0}\tau_{a_1}$, and for $i\in \Z$, $a_{2i}:=a_0^{\rho^i}$ and $a_{2i+1}:=a_1^{\rho^i}$. Since $\rho$ is an automorphism of $V$, for every $j\in \Z$,  $a_j$ is an axis. Denote by $\tau_j:=\tau_{a_j}$ the corresponding Miyamoto involution.

\begin{lemma}\label{invariances}
For every $n\in \N$,  and $i,j\in \Z$ such that $i\equiv j \: \bmod n$ we have
$$
a_ia_{i+n}-\beta(a_i+a_{i+n})=a_ja_{j+n}-\beta(a_j+a_{j+n}),
$$
\end{lemma}
\begin{proof}
This follows immediately from Lemma~\ref{invariant}.
\end{proof}

For $ n\in \N $ and $r\in \{0, \ldots , n-1\}$ set
\begin{equation}\label{s}
s_{r,n}:=a_ra_{r+n}-\beta (a_r+a_{r+n}). 
\end{equation}

If $\{0,1,\al, \bt\}$ are pairwise distinguishable in $R$, i.e. $\al$, $\bt$, $\al-1$, $\bt-1$, and $\al-\bt$ are invertible in $R$, by~\cite[Proposition~2.4]{FMS1}, for every $a\in \mathcal A$, there is a function $\lambda_a: V\to R$, such that every $v\in V$ can be written as
$$v=\lm_a(v)a+u \mbox{ with } u\in \bigoplus_{\delta\neq 1}V_\delta^a.
$$
\begin{remark}
Note that, when $R$ is a field, then $\{0,1,\al, \bt\}$ are always pairwise distinguishable in $R$. Since we are interested in algebras over fields, from now on we assume $0,1,\al, \bt$ are pairwise distinguishable in $R$. For the same reason, since the characteristic cannot be $2$, we also assume  that $2$ is invertible in $R$.
\end{remark}

For $i\in \Z$, let 
 \begin{equation}\label{ai}
 a_i=\lm_{a_0}(a_i) a_0+u_i+v_i +w_i
 \end{equation}
 be the decomposition of $a_i$ into $ad_{a_0}$-eigenvectors, where $u_i$ is a $0$-eigenvector, $v_i$ is an $\alpha$-eigenvector and $w_i$ is a $\beta$-eigenvector. 
 
  \begin{lemma}
\label{lambdaa}
With the above notation,   
\begin{enumerate}
\item $u_i=\frac{1}{\alpha} ((\lambda_{a_0}(a_i) - \beta - \alpha \lambda_{a_0}(a_i)) a_0 + \frac{1}{2}(\alpha - \beta) (a_i + a_{-i})-s_{0,i} )$;
\item $v_i=\frac{1}{\alpha} ((\bt-\lambda_{a_0}(a_i))a_0+\frac{\bt}{2}(a_i+a_{-i})+s_{0,i})$;
\item $w_i=\frac{1}{2}(a_i-a_{-i})$.
\end{enumerate}
\end{lemma}

\begin{lemma}\label{ideal}
Let $I$ be an ideal of $V$, $a$ an axis of $V$, $x\in V$  and let 
$$x=x_1+x_0+x_\alpha+x_\beta
$$ be the decomposition of $x$ as sum of $\ad_a$-eigenvectors. If $x\in I$, then $x_1, x_0, x_\alpha, x_\beta\in I$. Moreover,  $I$ is $\tau_a$-invariant.
\end{lemma}
\begin{proof}
Suppose $x\in I$. Then $I$ contains the vectors
\begin{eqnarray*}
x-ax&=&x_0+(1-\alpha)x_\alpha+(1-\bt)x_\bt,\\
a(x-ax)&=&\alpha(1-\alpha)x_\alpha+\bt(1-\bt)x_\bt,\\
a(a(x-ax))-\bt a(x-ax) &=&\alpha(\al-\bt)(1-\alpha)x_\alpha .
\end{eqnarray*}
Since, $0,1,\al, \bt$ are pairwise distinguishable in $R$, it follows that $I$ contains $x_1$, $x_0$, $x_\alpha$, $x_\beta$. Since $x^{\tau_a}=x_1+x_0+x_\alpha-x_\beta\in I$, the last assertion follows.
\end{proof}

\section{The universal object} \label{mult}
From now on we assume that $\al=2\bt$,  $\{1,0,2\bt, \bt\}$ is a  set of  pairwise distinguishable elements in $R$, and $2$ is invertible in $R$. 
Let  $\Vo$ be the universal $2$-generated primitive axial algebra over the ring $\Ro$ as defined in~\cite{FMS1} and let $\mathcal A:=\{ \ba_0, \ba_1\}$ be its  generating set of axes. That is, $\Ro$ and $\Vo$ are defined as follows
\begin{itemize}
\item[-] $D$ is the polynomial ring 
$${\Z}[x_i, y_i, w_i, z_{1,2}, t_1\:|\:i\in \{1, 2\}],$$
where $x_i, y_i, w_i, z_{1,2}, t_1$ are algebraically independent indeterminates over $\Z$, for $i\in \{1, 2\}$; 
\item[-] $L$ is the ideal of $D$ generated by the set 
$$ \Sigma:=\{x_iy_{i}-1, \: (1-x_i)w_{i}-1,\:2t_1-1, \:x_1-2x_2, z_{1,2}\:|\: i\in \{1, 2\}
\};$$ 
\item [-] $\hat D:=D/L$. For $d\in D$, we denote the element $L+d$ by $\hat d$. 
\item [-] $W$ is the free commutative magma generated by the elements of $\A$ subject to the condition that every element of $\A$ is idempotent;

\item[-] $\aR:={\hat D}[\Lambda]$ is the ring of polynomials with coefficients in $\hat D$ and  indeterminates set  $\Lambda:=\{\lambda_{ \bc,  \bw}\:|\: \bc \in \A,  \bw\in W, \bc\neq \bw\}$ where $\lambda_{ \bc,  \bw}=\lambda_{ \bc',  \bw'}$ if and only if $ \bc= \bc'$ and $ \bw= \bw'$.
\item[-] $\aV:=\aR[W]$ is the set of all formal linear combinations $\sum_{\bw\in W}\gamma_\bw \bw$ of the elements of $W$ with coefficients in $\aR$, with only finitely many coefficients different from zero.  Endow $\aV$ with  the usual structure of a commutative non associative $\aR$-algebra.
\end{itemize}

For $i\in \Z$, set  
$$
\lm_i:=\lm_{\ba_0}(\ba_i).
$$

By Corollary~3.8 in~\cite{FMS1}, the permutation that swaps $\ba_0$ and $\ba_1$ induces a  semi-automorphism $f$ of $\Vo$, i.e. $f$ acts on $\Ro$ as a ring homomorphism, on $\Vo$ as a (non-associative) ring homomorphism and  
$(\gamma \bv)^f=\gamma^f \bv^f$, for every $\gamma\in \Ro$ and $\bv \in \Vo$. By~\cite[Lemma~3.2]{FMS1},   
$$
\lambda_{\ba_1}(\ba_0) =\lambda_1^f, \mbox{ and }\:\:
 \lambda_{\ba_1}(\ba_{-1})= \lambda_2^f.
 $$
 
Set $G_0:=\langle \tau_0, \tau_1\rangle$ and $G:=\langle \tau_0, f\rangle$.  The following lemma is the analogue of~\cite[Lemma~4.5]{FMS1}.

\begin{lemma}\label{action}
The groups $G_0$ and $G$ are dihedral groups, $G_0$ is a normal subgroup of $G$ such that $|G:G_0|\leq 2$. 
For every $n\in \N$, the set $\{\bs_{0, n}, \ldots , \bs_{n-1,n}\}$ is invariant under the action of $G$. In particular, if $K_n$ is the kernel of this action, we have 
\begin{enumerate}
\item $K_1=G$;
\item $K_2=G_0$, in particular $\bs_{0,2}^f=\bs_{1,2}$; 
\item $G/K_3$ induces the full permutation group on the set $\{\bs_{0, 3}, \bs_{1, 3}, \bs_{2,3}\}$ with point stabilisers generated by $\tau_0K_3$, $\tau_1K_3$ and $fK_3$, respectively. In particular $\bs_{0,3}^f=\bs_{1,3}$ and $\bs_{0,3}^{\tau_1}=\bs_{2,3}$. 
\end{enumerate}  
\end{lemma}
\begin{proof}

This  follows immediately from the definitions.
\end{proof}

For $i,j\in \{1,2,3\}$, with the notation fixed before Lemma~\ref{lambdaa}, set
  $$
 P_{ij}:=\bu_i\bu_j + \bu_i\bv_j\:\: \mbox{ and } \:\:Q_{ij}:=\bu_i \bv_j+\frac{1}{\alpha^2}\bs_{0,i}\bs_{0,j}.
 $$
 
  \begin{lemma}
  \label{psi}
 For $i,j\in \{1,2,3\}$ we have
 \begin{equation}\label{ss}
 \bs_{0,i}\cdot \bs_{0,j}=-\alpha (\ba_0P_{ij}- \alpha Q_{ij}).
 \end{equation}
\end{lemma}
 \begin{proof}
Since $\bu_i$ and $\bv_j$ are a $0$-eigenvector and an $\alpha$-eigenvector for $\ad_{\ba_0}$, respectively,
 by the fusion rule, we have $\ba_0P_{ij}=\alpha(\bu_i\cdot \bv_j)$ and the result follows.  
\end{proof}

The following polynomial will play a crucial r\^ole in the classification of the non symmetric algebras in Section~\ref{ns}:
$$
Z(x,y):=\frac{2}{\bt}x+\frac{(2\bt-1)}{\bt^2}y-\frac{(4\bt-1)}{\bt}.
$$

\begin{lemma}
\label{a0s2f}
In $\Vo$ the following equalities hold:
\begin{eqnarray*}
\bs_{0,2}& =& -\frac{\beta}{2}(\ba_{-2}+\ba_2)+\bt Z(\lmu, \lmf)(\ba_1+\ba_{-1})\\
&-&\left [ 2 Z(\lmu, \lmf)(\lmu-\bt)-(\lmd-\bt) \right ]\ba_0+ 2Z(\lmu, \lmf)\bs_{0,1}.
\end{eqnarray*}
and 
\begin{eqnarray*}
\bs_{1,2}& =& -\frac{\beta}{2}(\ba_{-1}+\ba_3)+\bt Z(\lmf, \lmu)(\ba_0+\ba_2)\\
&-&\left [ 2 Z(\lmf, \lmu)(\lmf-\bt)-(\lmdf-\bt) \right ]\ba_1+ 2Z(\lmf, \lmu)\bs_{0,1}.
\end{eqnarray*}
\end{lemma} 
\begin{proof}
Since $\al=2\bt$, from the first formula in~\cite[Lemma~4.7]{FMS1} we deduce the expression for $\bs_{0,2}$. The expression for $\bs_{1,2}$ follows by applying $f$.
\end{proof}

\begin{lemma}\label{a3}
In $\Vo$ we have
\begin{eqnarray*}
\ba_4=\ba_{-2} &- &2Z(\lmu,\lmf)(\ba_{-1}-\ba_3)\\
&+&\frac{1}{\bt}\left [ 4Z(\lmu,\lmf) \left(\lmu-\bt \right)-\left (2\lmd-\bt\right ) \right ](\ba_{0}-\ba_2)
\end{eqnarray*}
and 
\begin{eqnarray*}
\ba_{-3}=\ba_{3}& -&2Z(\lmf,\lmu)(\ba_{2}-\ba_{-2})\\
& +&\frac{1}{\bt}\left [4Z(\lmf,\lmu)\left (\lmf-\bt\right )-\left (2\lmdf-\bt\right ) \right ](\ba_{1}-\ba_{-1}).
\end{eqnarray*}

\end{lemma} 
\begin{proof}
Since $\bs_{0,2}$ is invariant under $\tau_1$, we have $\bs_{0,2}-\bs_{0,2}^{\tau_1}=0$. On the other hand,  in the expression $\bs_{0,2}-\bs_{0,2}^{\tau_1}$ obtained from the first formula of Lemma~\ref{a0s2f}, the coefficient of $\ba_4$ is $-\bt/2$, which is invertible in $\Ro$. Hence we deduce the expression for $\ba_4$. By applying the map $f$ to the expression for $\ba_4$ we get the expression for $\ba_{-3}$.
 \end{proof}
 
 \begin{lemma} \label{s1s1}
In $\Vo$ we have
\begin{eqnarray*}
\lefteqn{\bs_{0,1}\bs_{0,1}=}\\
&& +\frac{\bt^2}{4} Z(\lmf,\lmu)(\ba_{-2}+\ba_2) \\
&&+\frac{1}{2}\left [-2(2\bt-1)(\lmu^2+{\lmf}^2)-\frac{(8\bt^2-4\bt+1)}{\bt}\lmu\lmf+(16\bt^2-7\bt+1)\lmu \right .\\
&& \:\:\:\:\:\:\left . +(14\bt^2-8\bt+1)\lmf-\bt(14\bt^2-7\bt+1)\right ](\ba_{-1}+\ba_{1})\\
&&+ \left [\frac{2(2\bt-1)}{\bt}\lmu^3+\frac{(8\bt^2-4\bt+1)}{\bt^2}\lmu^2\lmf+\frac{2(2\bt-1)}{\bt}\lmu{\lmf}^2-(18\bt-6)\lmu^2 \right . \\
&& \:\:\:\:\:\: \left . -\frac{2(10\bt^2-5\bt+1)}{\bt}\lmu\lmf-2(\bt-1){\lmf}^2-\frac{(2\bt-1)}{2}\lmu\lmd-\bt\lmf\lmd \right .\\
&& \:\:\:\:\:\:\left . +\frac{(54\bt^2-17\bt+1)}{2}\lmu+(9\bt^2-6\bt+1)\lmf+\frac{\bt (5\bt-1)}{2}\lmd-\frac{\bt^2}{2}\lmdf \right .\\
&& \:\:\:\:\:\:\left . -\frac{\bt(24\bt^2-9\bt+1)}{2}\right ]\ba_0\\
&& +\left [-\frac{2(2\bt-1)}{\bt}\lmu^2-\frac{(6\bt^2-3\bt+1)}{\bt^2}\lmu\lmf-\frac{2(\bt-1)}{\bt}{\lmf}^2+\frac{(16\bt^2-7\bt+1)}{\bt}\lmu \right . \\
&& \:\:\:\:\:\:\: \left . +\frac{(10\bt^2-7\bt+1)}{\bt}\lmf-\frac{\bt}{2}\lmdf-\frac{(57\bt^2+26\bt-4)}{4}\right ]\bs_{0,1}\\
&& +\frac{\bt^2}{4}\bs_{0,3}.
\end{eqnarray*} 
 \end{lemma}
 \begin{proof}
By Lemma~\ref{a0s2f}, Lemma~\ref{a3}, and Lemma~4.3 in~\cite{FMS1}, we may compute the expression on the right hand side of the formula in Lemma~\ref{psi}, with $i=j=1$, and the   result follows.
  \end{proof}

\begin{lemma} \label{s3f}
In $\Vo$ we have
\begin{eqnarray*}
\lefteqn{\bs_{1,3}=\bs_{0,3}+ \bt Z(\lmf, \lmu)\ba_{-2}-\bt Z(\lmu,\lmf)\ba_3 } \\
&&+ \frac{1}{\bt^3}\left [-4\bt(2\bt-1)(\lmu^2+{\lmf}^2)-2(8\bt^2-4\bt+1)\lmu\lmf+2\bt(15\bt^2-7\bt+1)\lmu \right .\\
&&\left . \:\:\:\:\:\:\:+\bt(26\bt^2-15\bt+2)\lmf-\bt^2(24\bt^2-13\bt+2)\right ]\ba_{-1} \\
&&+ \frac{1}{\bt^4}\left [8\bt(2\bt-1)\lmu^3+4(8\bt^2-4\bt+1)\lmu^2\lmf+8\bt(2\bt-1)\lmu{\lmf}^2\right .\\
&&\:\:\:\:\:\:\left . -4\bt^2(15\bt-5)\lmu^2-2\bt(32\bt^2-16\bt+3)\lmu\lmf+4\bt^2{\lmf}^2-2\bt^2(2\bt+1)\lmu\lmd \right .\\
&& \left . \:\:\:\:\:\:-4\bt^3\lmf\lmd+2\bt^3(40\bt-9)\lmu+2\bt^2(2\bt^2-5\bt+1)\lmf+2\bt^3(5\bt-1)\lmd \right .\\
&& \left . \:\:\:\:\:\:-2\bt^4\lmdf-4\bt^4(5\bt-1)\right ]\ba_0\\
&&+ \frac{1}{\bt^4}\left [ -8\bt(2\bt-1)\lmu^2\lmf-4(8\bt^2-2\bt+1)\lmu{\lmf}^2-8\bt(2\bt-1){\lmf}^3-4\bt^2\lmu^2\right .\\
&&\left . \:\:\:\:\:\:+2\bt(32\bt^2-16\bt+3)\lmu\lmf+4\bt^2(16\bt-5){\lmf}^2+4\bt^3\lmu\lmdf \right . \\
&& \left. \:\:\:\:\:\:\:+2\bt^2(2\bt-1)\lmf\lmdf-2\bt^2(2\bt^2+5\bt+1)\lmu-2\bt^3(40\bt-9)\lmf+2\bt^4\lmd \right .\\
&& \left . \:\:\:\:\:\:-2\bt^3(5\bt-1)\lmdf+4\bt^4(5\bt-1)\right ]\ba_1
\end{eqnarray*} 
\begin{eqnarray*}
&& +\frac{1}{\bt^3}\left [ 4\bt(2\bt-1)\lmu^2+2(8\bt^2-4\bt+1)\lmu\lmf+\bt(8\bt-4){\lmf}^2\right .\\
&& \left . \:\:\:\:\:-\bt(26\bt^2-15\bt+2)\lmu-2\bt(15\bt^2-7\bt+1)\lmf+\bt^2(24\bt^2-13\bt+2)\right ]\ba_2\\
&& +\frac{1}{\bt^2}\left [-8(\lmu^2-{\lmf}^2)+24\bt(\lmu-\lmf)+2\bt(\lmd-\lmdf)\right ] \bs_{0,1}.
\end{eqnarray*} 
Similarly, $\bs_{2,3}$ belongs to the linear span of the elements $\ba_{-2}$, $\ba_{-1}$, $\ba_0$, $\ba_1$, $\ba_2$, $\ba_3$, $\bs_{0,1}$, and $\bs_{0,3}$.
 \end{lemma}
 \begin{proof}
 Since, by Lemma~\ref{action},  $\bs_{0,1}$ is invariant under $f$, we have $\bs_{0,1}\bs_{0,1}-(\bs_{0,1}\bs_{0,1})^{f}=0$. Comparing the expressions for $\bs_{0,1}\bs_{0,1}$ and $(\bs_{0,1}\bs_{0,1})^{f}$ obtained from Lemma~\ref{s1s1}, we deduce the expression for $\bs_{1,3}$. By applying the map $\tau_0$ to the expression for $\bs_{1,3}$ we get the expression for $\bs_{2,3}$ and the last assertion follows from Lemma~\ref{a3}.
  \end{proof}

As a consequence of the {\it resurrection principle}~\cite[Lemma~1.7]{IPSS10}, we can now prove the following result.  We use double angular brackets to denote algebra generation and singular angular brackets for linear span.
\begin{proposition}\label{span}
$\Vo=\langle \ba_{-2}, \ba_{-1}, \ba_0, \ba_1, \ba_2, \ba_3, \bs_{0,1}, \bs_{0,3}\rangle$.
\end{proposition}
\begin{proof}
Set $U:=\langle \ba_{-2}, \ba_{-1}, \ba_0, \ba_1, \ba_2, \ba_3, \bs_{0,1}, \bs_{0,3}\rangle$. By Lemma~\ref{a3}, $\ba_4, \ba_{-3}\in U$, and  by Lemma~\ref{s3f} also $\bs_{1,3}$ and $\bs_{2,3}$ belong to $U$. It follows that $U$ is invariant under the maps $\tau_0$, $\tau_1$,  and $f$. Hence, $a_i$ belongs to $U$, for every $i\in \Z$. Now we show that $U$ is closed under the algebra product. Since it is invariant under the maps $\tau_0$, $\tau_1$,  and $f$, it is enough to show that it is invariant under the action of $\ad_{\ba_0}$ and it contains $\bs_{0,1}\bs_{0,1}$, $\bs_{0,3}\bs_{0,3}$, and $\bs_{0,1}\bs_{0,3}$. The products $\ba_0\ba_i$, for $i\in \{-2,-1,0,1,2,3\}$ belong to $U$ by the definition of $U$ and by Lemma~\ref{a0s2f}.  By~\cite[Lemma~4.3]{FMS1}, $U$ contains  $\ba_0\bs_{0,1}$ and $\ba_0\bs_{0,3}$. The product $\bs_{0,1}\bs_{0,1}$ belongs to $U$ by Lemma~\ref{s1s1} and similarly, by Lemma~\ref{lambdaa} and Lemma~\ref{psi}, the products $\bs_{0,3}\bs_{0,3}$, and $\bs_{0,1}\bs_{0,3}$ belong to $U$.
Hence $U$ is a subalgebra of $\Vo$ and, since it contains the generators $\ba_0$ and $\ba_1$, we get $U=\Vo$. \end{proof}
 
 \begin{remark}\label{struct}
Note that the above proof gives a constructive way to compute the structure constants of the algebra $\Vo$ relative to the generating set $B$. This has been done with the use of GAP~\cite{GAP}. The explicit expressions however are far too long to be written here.
\end{remark}

\begin{cor}\label{Rin}
 $\Ro$ is generated as a $\hat D$-algebra by $\lambda_1$, $\lambda_2$,  $\lambda_1^f$, and $\lambda_2^f$.
\end{cor}

\begin{proof}
By Proposition~\ref{span}, $\Vo$ is generated as  $\Ro$-module by the set $B:=\{\ba_{-2}, \ba_{-1},$ $ \ba_0, \ba_1, \ba_2, \ba_3, \bs_{0,1}, \bs_{0,3}\}$. Since $\lambda_{\ba_1}(v)=(\lambda_{\ba_0}(\bv^f))^f$, $\lambda_{\ba_0}$ is a linear function, and $\Ro=\Ro^f$,  we just need to show that $\lambda_{\ba_0}(\bv) \in \hat D[\lmu, \lmf, \lmd, \lmdf]$ for every $\bv\in B$. 
By definition we have 
$$\lambda_{\ba_0}(\ba_0)=1, \:\:\lambda_{\ba_0}(\ba_1)=\lambda_1,\:\:\lambda_{\ba_0}(\ba_2)=\lambda_2, \mbox{ and }\:\:\lambda_{\ba_0}(\ba_3)=\lambda_3.$$ Since $\tau_0$  fixes $\ba_0$ and is an $\Ro$-automorphism of $\Vo$, we get
$$\lambda_{\ba_0}(\ba_{-1})=\lambda_{\ba_0}((\ba_{1})^{\tau_0})=\lambda_1,$$
$$\lambda_{\ba_0}(\ba_{-2})=\lambda_{\ba_0}((\ba_{2})^{\tau_0})=\lambda_2,$$
 and 
$$
\lambda_{\ba_0}(\bs_{0,1})=\lambda_{\ba_0}(\ba_0\ba_1-\beta \ba_0-\beta \ba_1)=\lambda_1-\beta-\beta \lambda_1,
$$
and
$$
\lambda_{\ba_0}(\bs_{0,3})=\lambda_{\ba_0}(\ba_0\ba_3-\beta \ba_0-\beta \ba_3)=\lambda_3-\beta-\beta \lambda_3.
$$
We conclude the proof by showing that $\lambda_3\in \hat D[\lmu, \lmf, \lmd, \lmdf]$. Set 
$$\mathbf{\phi}:= \bu_1\bu_1-\bv_1\bv_1- \lambda_{\ba_0}(\bu_1\bu_1-\bv_1\bv_1)\ba_0
$$
and 
$$
\mathbf{z}:=\mathbf{\phi}-2(2\bt-\lmu)\bu_1.
$$
Then, by the fusion law, $\phi$ is a $0$-eigenvector for $\ad_{\ba_0}$ and so $\mathbf z$ is a $0$-eigenvector for $\ad_{\ba_0}$ as well. Since $\bs_{0,1}$ is $\tau_0$-invariant, it lies in $\Vo^{\ba_0}_+$ and the fusion law implies that the product $\mathbf{z}\bs_{0,1}$ is a sum of a $0$- and an $\al$-eigenvector for $\ad_{\ba_0}$. In particular, $\lambda_{\ba_0}(\mathbf{z}\bs_{0,1})=0$. By Remark~\ref{struct} we can compute explicitly  the product $\mathbf{z}\bs_{0,1}$:

\begin{eqnarray*}
\lefteqn{ \mathbf{z}\bs_{0,1}=-\frac{\bt^3}{4}\ba_3}\\
 && +\frac{\bt}{4}\left [2\bt\lmu-\lmf-\bt(\bt-1)\right ]\ba_{-2}\\
&& +\left [-\bt^2\lmu^2-\frac{(2\bt^2+\bt-1)}{2}\lmu\lmf-\frac{(2\bt^2-4\bt+1)}{2\bt}{\lmf}^2+\frac{(4\bt^3-\bt^2-\bt)}{2}\lmu \right .\\
&&\left . +(2\bt^2-4\bt+1)\lmf+\frac{\bt^3}{4}\lmd-\frac{\bt^2}{4}\lmdf+\frac{\bt(4\bt-1}{2}\right ]\ba_{-1}\\
&& +\left [
2(2\bt-1)\lm^3+\frac{(2\bt-1)^2}{\bt}\lmu^2\lmf-(10\bt^2-8\bt+1)\lmu^2+(-2\bt^2+3\bt-1)\lmu\lmf \right .\\
&&\left .-\frac{\bt(2\bt-1)}{2}\lmu\lmd+\bt(2\bt-1)^2\lmu+\frac{\bt(2\bt-1)}{2}\lmf+\frac{\bt^2(\bt-1)}{2}\lmd-\frac{\bt^2(4\bt-1)}{2}\right ]\ba_{0}\\
&& +\left [
-\bt^2)\lmu^2-\frac{(\bt+3)(2\bt-1)}{2}\lmu\lmf-\frac{(6\bt^2-4\bt+1)}{2\bt}{\lmf}^2+\frac{\bt(4\bt^2+3\bt-3)}{2}\lmu \right .\\
&&\left .+(8\bt^2-5\bt+1)\lmf+\frac{bt^3}{4}\lmd+\frac{\bt^2}{4}\lmdf-\frac{\bt(17\bt^2-12\bt+2)}{4}\right ]\ba_{1}\\
&& +\left [
\frac{\bt(3\bt-1)}{2}\lmu+\frac{\bt(4\bt-1)}{4}\lmf-\frac{3\bt^2(3\bt-1)}{4}\right ]\ba_{2}\\
&& +\left [
-2\bt\lmu^2-(2\bt-1)\lmu\lmf+\bt(4\bt+1)\lmu+(2\bt-1)\lmf+\frac{\bt^2}{2}\lmd-\frac{\bt(9\bt+2)}{2}\right ]\bs_{0,1}.
\end{eqnarray*}

Since $\lambda_{a_0}(\mathbf{z}\bs_{0,1})=0$, taking the image under $\lambda_{\ba_0}$ of both sides, we get
\begin{eqnarray*}
\lambda_3&=&\frac{8(\bt-1)}{\bt^3}\lmu^3-\frac{4(2\bt^2+\bt-1)}{\bt^4}\lmu^2\lmf-\frac{4(2\bt-1)^2}{\bt^4}\lmu{\lmf}^2\\
&&-\frac{4(4\bt^2-7\bt+1)}{\bt^3}\lmu^2+\frac{16(2\bt-1)}{\bt^2}\lmu{\lmf}+\frac{6}{\bt}\lmu\lmd+\frac{2(2\bt-1)}{\bt^2}\lmf\lmd\\
&&+\frac{(\bt^2-22\bt+4)}{\bt^2}\lmu-\frac{2(2\bt-1)}{\bt^2}\lmf-\frac{2(5\bt+1)}{\bt}\lmd+\frac{2(5\bt-1)}{\bt}.
\end{eqnarray*}
\end{proof}

 We conclude this section giving some relations in $\Vo$ and $\Ro$ which will be useful in the sequel for the classification of the algebras.
 Set
 $$
 \mathbf{d}_1:=\bs_{2,3}^f-\bs_{2,3}, \:\:
\mathbf{d}_2:={ \mathbf{d}_1}^{\tau_1},\:\: \mbox{ and, for } i\in \{1,2\},\:\:
\mathbf{D}_i:={\mathbf{d}_i}^{\tau_0}-\mathbf{d}_i;
$$
 $$
\mathbf{e}:= \bu_1^{\tau_1} \bv_3^{\tau_1} \:\:\mbox{ and }\:  \mathbf{E}:= \ba_2 \mathbf{e}-2\bt  \mathbf{e}.
$$

 \begin{lemma}\label{u1u2}
The following identities hold in $\Vo$, for $i\in \{1,2\}$:
\begin{enumerate}
 \item 
 $\mathbf{d}_ i=0, \:\:, \mathbf{D}_i=0, \:\: \mathbf{E}=0$;
 \item there exists an element $t(\lmu, \lmf, \lmd, \lmdf) \in \Ro$ such that 
\begin{equation*}
t_1(\lmu, \lmf, \lmd, \lmdf)\ba_0+
 \frac{2}{\bt}(\lmu-\lmf)\left [\bt\lmu+(\bt-1)(\lmf-\bt)\right ](\ba_{-1}+\ba_1+\frac{2}{\bt}\bs_{0,1})=0.
\end{equation*} 
\end{enumerate}
\end{lemma}
\begin{proof}
Identities involving the $\mathbf{d}_i$'s and $\mathbf{D}_i$'s follow from Lemma~\ref{action}.
By the fusion law, the product $\bu_1\bu_2$ is a $0$-eigenvector for $\ad_{\ba_0}$ and the product $\bu_1^{\tau_1} \bv_3^{\tau_1}$ is a $2\bt$-eigenvector for  $\ad_{\ba_2}$. The last claim follows by an explicit computation of the product $\ba_0(\bu_1\bu_2)$, which gives the left hand side of the equation.
\end{proof}

\begin{lemma}\label{polinomi}
In the ring $\Ro$ the following  holds:
\begin{enumerate}
\item $\lambda_{\ba_0}(\ba_4\ba_4-\ba_4)=0$,
\item $\lambda_{\ba_0}(\mathbf{d}_1)=0$,
\item $\lambda_{\ba_0}(\mathbf{d}_2)=0$,
\item $\lambda_{\ba_1}(\mathbf{d}_1)=0$.
\end{enumerate}
\end{lemma}
\begin{proof}
The first equation follows from the fact that $\ba_4$ is an idempotent. The remaining follow from Lemma~\ref{u1u2}.
 \end{proof}

 \section{Proof of Theorem~\ref{nec}}\label{strategy}
 
 By Remark~\ref{struct}, the four expressions on the left hand side of the identities in Lemma~\ref{polinomi} can be computed explicitly and produce respectively four polynomials $p_i(x,y,z,t)$ for $i\in \{1, \ldots , 4\}$ in $\hat D [x,y,z,t]$ (with $x,y,z,t$ indeterminates on $\hat D$), that simultaneously annihilate on the quadruple $(\lambda_1, \lambda_1^f, \lambda_2, \lambda_2^f)$.  We define also, for $i\in \{1,2,3\}$, $q_i(x,z):=p_i(x,x,z,z)$. The polynomials $p_i$'s are too long to be displayed here but have been computed using~\cite{Singular} or~\cite{GAP}, while the polynomials $q_i$ are the following
{\small
 \begin{eqnarray*}
\lefteqn{q_1(x,z)=\frac{128}{\bt^{10}}(-384\bt^5+608\bt^4-376\bt^3+114\bt^2-17\bt+1)x^7}\\
&& +\frac{64}{\bt^{10}}(4352\bt^6-6080\bt^5+2992\bt^4-516\bt^3-40\bt^2+23\bt-2)x^6\\
&& +\frac{64}{\bt^8}(64\bt^4-96\bt^3+52\bt^2-12\bt+1)x^5z\\
&&+\frac{16}{\bt^9}(-38720\bt^6+42912\bt^5-9252\bt^4-5928\bt^3+3477\bt^2-660\bt+44)x^5\\
&&+\frac{16}{\bt^8}(-3168\bt^5+4832\bt^4-2782\bt^3+747\bt^2-92\bt+4)x^4z\\
&&+\frac{32}{\bt^5}(8\bt^2-6\bt+1)x^3z^2\\
&&+\frac{8}{\bt^8}(84832\bt^6-48224\bt^5-50482\bt^4+55573\bt^3-20164\bt^2+3262\bt-200)x^4\\
&&+\frac{8}{\bt^7}(19792\bt^5-30292\bt^4+17700\bt^3-4917\bt^2+647\bt-32)x^3z\\
&&+\frac{16}{\bt^5}(-72\bt^3+62\bt^2-15\bt+1)x^2z^2\\
&& +\frac{8}{\bt^7}(-45888\bt^6-33584\bt^5+119184\bt^4-85132\bt^3+27054\bt^2-4089\bt+240)x^3\\
&&+\frac{4}{\bt^6}(-52880\bt^5+81156\bt^4-47828\bt^3+13527\bt^2-1838\bt+96)x^2z\\
&&+\frac{32}{\bt^4}(48\bt^3-44\bt^2+12\bt-1)xz^2+\frac{4}{\bt^2}(2\bt-1)z^3\\
&&+\frac{4}{\bt^6}(19648\bt^6+114384\bt^5-204648\bt^4+128262\bt^3-38411\bt^2+5598\bt-320)x^2\\
&&+\frac{8}{\bt^5}(16288\bt^5-25096\bt^4+14904\bt^3-4272\bt^2+593\bt-32)xz\\
\end{eqnarray*}

 \begin{eqnarray*}
&&+\frac{2}{\bt^3}(-322\bt^3+301\bt^2-86\bt+8)z^2\\
&&+\frac{8}{\bt^5}(-26112\bt^5+40040\bt^4-23878\bt^3+6959\bt^2-995\bt+56)x\:\:\:\:\:\:\:\:\:\:\:\:\:\:\:\:\:\:\:\:\\
&&+\frac{2}{\bt^4}(-15264\bt^5+23658\bt^4-14169\bt^3+4110\bt^2-580\bt+32)z\\
&&+\frac{4}{\bt^4}(7632\bt^5-11668\bt^4+6932\bt^3-2011\bt^2+286\bt-16),
\end{eqnarray*}

 \begin{eqnarray*}
\lefteqn{q_2(x,z)= \frac{-8(8\bt^2-6\bt+1)}{\bt^4}x^4+\frac{(160\bt^3-56\bt^2-28\bt+8)}{\bt^4}x^3+\frac{(8\bt-4)}{\bt^2}x^2z}\\
&& -\frac{(96\bt^3+96\bt^2-112\bt+20)}{\bt^3}x^2-\frac{(44\bt^2-30\bt+4)}{\bt^2}xz\\
&&+\frac{(140\bt^2-102\bt+16)}{\bt^2}x+\frac{(36\bt^2-26\bt+4)}{\bt}z\\
&&-\frac{(36\bt^2-26\bt+4)}{\bt},\\
\lefteqn{q_3(x,z)=\frac{(-128\bt^3+160\bt^2-64\bt+8)}{\bt^5}x^4+\frac{(64\bt^2-48\bt+8)}{\bt^4}x^3z}\\
&&+\frac{(288\bt^4-280\bt^3+20\bt^2+40\bt-8)}{\bt^5}x^3+\frac{(-112\bt^3+48\bt^2+12\bt-4)}{\bt^4}x^2z\\
&& -\frac{(8\bt-4)}{\bt^2}xz^2+\frac{(-160\bt^4+8\bt^3+228\bt^2-136\bt+20)}{\bt^4}x^2\\
&&+\frac{(12\bt^3+70\bt^2-54\bt+8)}{\bt^3}xz+\frac{(8\bt-4)}{\bt}z^2+\frac{(148\bt^3-246\bt^2+118\bt-16)}{\bt^3}x\\
&&+\frac{(36\bt^3-70\bt^2+34\bt-4)}{\bt^2}z+\frac{(-36\bt^3+62\bt^2-30\bt+4)}{\bt^2}.
\end{eqnarray*}
}
We can now prove  Theorem~\ref{nec}.

\begin{proof}[Proof of Theorem~\ref{nec}]
Let $V$ be a  primitive axial algebra of Monster type $(2\bt, \beta)$ over a ring $R$ of characteristic other than $2$, generated by the two axes $a_0$ and $a_1$. Then, by~\cite[Corollary~3.8]{FMS1}, $V$ is a homomorphic image of $\Vo\otimes_{\hat D} R$ and $R$ is a homomorphic image of $\Ro\otimes_{\hat D} R$. We identify the elements of $\hat D$ with their images in $R$ so that the polynomials  $p_i$ and $q_i$  defined above are considered as polynomials in $R[x,y,z,t]$ and $R[x,z]$, respectively. For each $i\in \Z$, let $ a_i$ be the image of the axis $\ba_i$ and let 
$$
 P:=(\lambda_{ a_0}(a_1), \lambda_{ a_1}( a_0), \lambda_{ a_0}( a_2), \lambda_{a_1}( a_{-1})).
 $$
 By Corollary~\ref{Rin}, $\Ro\otimes_{\hat D} \F$ is 
isomorphic to $\F[\lambda_1, \lambda_1^f, \lambda_2, \lambda_2^f]$. Let 
$$
v_{P}:\F[x,y,z,t] \to \F
$$ 
be the $\F$-algebra homomorphism that associates to each polynomial $f \in \F[x,y,z,t] $ its value on the quadruple $P$ (of course we are assuming $x,y,z,t$ to be algebraically independent indeterminates over $\F$ too). Let $I_{P}$ be the kernel of $v_{P}$ and set  $$
U_{P}:=\Vo\otimes_{\hat D} \F/(\Vo\otimes_{\hat D} \F)I_P.
$$
Then, $U_{P}$ is a primitive axial $\F$-algebra of Monster type $(2\beta, \beta)$.  
We denote the images of an element $\delta$ of $\Ro\otimes_{\hat D} \F$ in $\F$ via $v_{P}$ by $\bar \delta$ and by $\overline{p}_i$ and $\overline{q}_i$ the polynomials in $\F[x,y,z,t]$ and $\F[x,z]$ corresponding to $p_i$ and $q_i$, respectively. Set 
\begin{equation}\label{T}
T:=\{\overline{p}_1, \overline{p}_2, \overline{p}_3, \overline{p}_4\}.
\end{equation}
By definition, the $\overline{p}_i$'s have the coefficients in the field $ \F_0(\bt)$. By~\cite[Corollary~3.8]{FMS1}, $P$ is the homomorphic image in $R^4$ of the quadruple $(\lambda_1, \lambda_1^f, \lambda_2, \lambda_2^f)$ and so it is a common zero of the set $T$ and $V$ is isomorphic to a quotient of $U_{P}$. Hence the map 
$\xi: \mathcal M_{2}(2,\F) \to \mathcal V(T)$ mapping $V$ to $P$ gives us the map of Theorem~\ref{nec}.
\end{proof}

If $V$ satisfies the hypothesis of Theorem~\ref{nec} and, in addition, it is symmetric, then 
$$\lambda_{ a_0}( a_1)= \lambda_{ a_1}( a_0) \:\:\mbox{ and }\:\: 
\lambda_{ a_0}( a_2),=\lambda_{ a_1}( a_{-1})
$$
and the pair $ (\lambda_{ a_0}( a_1), \lambda_{ a_0}( a_2))$ is a solution of the system 
\begin{eqnarray}\label{syst}
\left \{
\begin{array}{rcl}
\bar q_1(x,z)& =&0\\
\bar q_2(x,z)& =&0\\
\bar q_3(x,z)& =&0.\\
\end{array}
\right . 
\end{eqnarray}
\medskip

\begin{lemma}\label{res}
For any field $\F$, the resultant of the polynomials $q_2(x,z)$ and 
$q_3(x,z)$ with respect to $z$ is 
$$
\gamma  x(x-1)(2x-\bt)^3[(16\bt-6)x+(-18\bt^2+\bt+2)] [(8\bt-2)x+(-9\bt^2+2\bt)],
$$
where 
$$\gamma:=\frac{-16(2\bt-1)^3(4\bt-1)}{\bt^{10}}.
$$

\end{lemma}
\begin{proof}
The resultant have been computed in the ring $\Z[\bt, \bt^{-1}][x]$ using~\cite{Singular}.
\end{proof}

We set 
\begin{eqnarray*}
\mathcal S_0&:=&\left \{\left (\frac{\bt}{2}, \frac{\bt}{2}\right ),\:\: \left (\bt, 0 \right ), \:\:  \left (\bt, \frac{\bt}{2} \right ) \right \},\\
\mathcal S_1&:=&\mathcal S_0\cup \left \{(1,1),\:\: (0,1),\:\: \left (\bt, 1 \right )\right \},\\
\mathcal S_2&:=&\mathcal S_1\cup \left \{\left (\frac{(18\bt^2-\bt-2)}{2(8\bt-3)}, \frac{(48\bt^4-28\bt^3+7\bt-2)(3\bt-1)}{2\bt^2(8\bt-3)^2}\right )\right \} \mbox{ if }\bt\neq \frac{3}{8},\\
\mathcal S_3&:=&\mathcal S_2\cup \left \{\left (\frac{(9\bt^2-2\bt)}{2(4\bt-1)}, \frac{(9\bt^2-2\bt)}{2(4\bt-1)}\right )\right \} \mbox{ if } \bt\neq \frac{1}{4}.
\end{eqnarray*}
\medskip

\begin{lemma}\label{solutions}
Let $\F$ be a field of characteristic other than $2$, $\bt\in \F$. Then, the set of the solutions of the system of equations~(\ref{syst}) is 
\begin{enumerate}
\item $\mathcal S_0\cup \left \{(\mu,1)\:|\:\mu\in \F\right \}$, if $\bt=\frac{1}{4}$; 
\item $\mathcal S_2\cup \mathcal S_3$, if either $\bt \in \{\frac{1}{2}, \frac{1}{3}, \frac{2}{7}\}$ or $\bt\not \in \left \{ \frac{1}{4}, \frac{3}{8}\right \}$ and 
\begin{equation}\label{condition1}
(16\bt^4-48\bt^3-51\bt^2+46\bt-8)(18\bt^2-\bt-2)(5\bt^2+\bt-1)(4\bt^2+2\bt-1)=0 \nonumber;
\end{equation}  
\item  $\mathcal S_3$ in all the remaining cases.
\end{enumerate}
\end{lemma}
\begin{proof}
Using~\cite{Singular}, it is straightforward to check that  the possible values $x_0$ for a solution $(x_0, z_0)$ of the system~(\ref{syst}) are given by Lemma~\ref{res} when $\bt\neq \frac{1}{4}$ and can be computed directly when $\bt= \frac{1}{4}$. Since $q_2(x,z)$ is linear in $z$, for every value of $x$ there is at most a solution of the system. The elements of $\mathcal S_2$ are indeed solutions.  When $x_0=\frac{(18\bt^2-\bt-2)}{2(8\bt-3)}$, we solve $q_2(x_0,z)=0$ obtaining the given corresponding value for $z_0$. On the other hand, the value of $q_1(x,z)$ computed in $\Z[\bt]$ on this pair is a non-zero polynomial in $\bt$ which vanishes exactly when either $\bt\in \{\frac{1}{2}, \frac{1}{3}, \frac{2}{7}\}$, or $\bt\neq \frac{3}{8}$ and the equation in the second claim is satisfied.
  \end{proof}

In order to classify primitive axial algebras of Monster type $(2\beta, \beta)$ over $\F$ generated by two axes $ a_0$ and $ a_1$ we can proceed, similarly as we did in~\cite{FMS1}, in the following way. We first solve the system~(\ref{syst}) and classify all symmetric algebras. Then we observe that, the even subalgebra $\langle \langle  a_0,  a_2\rangle \rangle$ and the odd  subalgebra $\langle \langle  a_{-1},  a_1\rangle \rangle$ are symmetric, since the automorphisms  $\tau_1$ and $\tau_0$ respectively, swap their generating axes. Hence, from the classification of the symmetric case, we know all possible configurations for the subalgebras $\langle \langle  a_0,  a_2\rangle \rangle$ and $\langle \langle  a_{-1},  a_1\rangle \rangle$ and from the relations found in Section~\ref{mult}, we derive the structure of the entire algebra.

\section{The symmetric case} \label{symmetric}
In this and in the following section we let $V$ be a primitive axial algebra of Monster type $(2\bt, \bt)$ over a field $\F$ of characteristic other than $2$, generated by the two axes $ a_0$ and $ a_1$. By~\cite[Corollary~3.8]{FMS1}, $V$ is a homomorphic image of $\Vo\otimes_{\hat D} \F$. For every element $\bv\in \Vo$, we denote by $ v$ its image in $V$. In particular $a_0$ and $a_1$ are the images of $\ba_0$ and $\ba_1$ and all the formulas obtained from the ones in Lemmas~\ref{psi},~\ref{a0s2f},~\ref{a3}, ~\ref{u1u2} with $a_i$ and $s_{r,j}$ in the place of $\ba_i$ and $\bs_{r,j}$, respectively, hold in $V$. With an abuse of notation we identify the elements of $\Ro$ with their images in $\F$, so that in particular $\lmu=\lambda_{ a_0}( a_1)$, $\lmf=\lambda_{ a_1}( a_0)$,  $\lmd=\lambda_{a_0}( a_2)$ , and $\lmdf= \lambda_{ a_1}( a_2)$. 

We begin with a quick overview of the known $2$-generated primitive symmetric algebras of Monster type $(2\bt, \bt)$. Among these, there are 
\begin{enumerate}
\item the algebras of Jordan type $1A$, $2B$, $3C(\bt)$ and $3C(2\bt)$. 
\item the algebra $3A(2\bt, \bt)$ defined in~\cite{R}.
\item  the algebras $4J(2\bt,\bt)$ and $6J(2\bt,\bt)$ defined in~\cite{DM}\footnote{We adopt here the terminology proposed in~\cite{forbidden}.}. 
\item the $3$-dimensional  algebra $V_3(\bt)$ with basis $( a_0,  a_1,  a_2)$ and the multiplication defined as in Table~\ref{tableV3}. Note that it coincides with the algebra $III_3(\xi, \frac{1-3\xi^2}{3\xi-1},0)^\times$ defined by Yabe in~\cite{Yabe}, with $\xi=2\bt$.
\item the $5$-dimensional algebra $4Y(2\bt,\bt)$ with basis $( a_{3},  a_0,  a_1,  a_2,  s)$ and multiplication table Table~\ref{tableY5}. Note that it coincides with the algebra $IV_2(\xi, \beta, \mu)$ defined by Yabe~\cite{Yabe}, when $\beta=\frac{1-\xi^2}{2}$ and $\xi=2\beta$. 
\end{enumerate}
\begin{table}

{\Small
  $$
\begin{array}{c|c|c}
\mathrm{ Basis } & \mathrm{ Products } & \mathrm{ Form }\\
\hline
 a_{-1},   a_0,  a_1& \begin{array}{l}  a_i^2= a_i\\
 a_i a_{i+1}=\frac{3}{2}\bt( a_i+ a_{i+1})+\frac{\bt}{2} a_{i-1} \:\:\:\:\:\mathrm{ where } \: a_i:= a_{i+3} \\
\end{array}&
\begin{array}{l}
( a_i,  a_i)=1\\
( a_i,  a_{i+1})=\frac{9\bt+1}{4}\\
\end{array} \\
\hline
\end{array} 
$$
}

  \caption{Multiplication table for the algebra $V_3(\bt)$}\label{tableV3} 
  \end{table}
 
 It is immediate that the values of $ (\lmu, \lmd)$ corresponding to the trivial algebra $1A$ and to the algebra $2B$ are $(1,1)$ and $(0,1)$, respectively.   In the  following lemmas we list the key features of these algebras.
 
\begin{lemma}\label{algebraV3}
Let $\F$ be a field of characteristic other than $2$ and $\bt\in \F$ such that $18\bt^2-\bt-1=0$. The algebra $V_3(\bt)$ is a $2$-generated symmetric Frobenius axial algebra satisfying the fusion law $\mathcal M(2\bt,\bt)$ and such that for every $i\in \{-1,0,1\}$, $\ad_{ a_i}$ has eigenvalues $1$, $2\bt$, and $\bt$. In particular it not an axial algebra of Jordan type. Moreover, $\lmu=\lmd=\frac{9\bt+1}{4}$. Furthermore 
\begin{enumerate}
\item if $ch\:\F\neq 3$, the algebra $V_3(\bt)$ is primitive and simple;
\item if $ch\:\F= 3$, then $\bt= 2$ and $V_3(\bt)$ is neither primitive nor simple. It has a $2$-dimensional quotient over the ideal $\F(  a_0+ a_{-1}+ a_2)$ isomorphic to $3C(-1)^\times$ and a quotient isomorphic to $1A$ (over the ideal $\langle  a_0- a_1,  a_0- a_2\rangle $). 
\end{enumerate}
\end{lemma} 
\begin{proof}
If $ch\:\F\neq 3$, then $ a_1- a_{-1}$ and $-\frac{3\bt+1}{8}  a_0+\frac{\bt}{2}( a_1+ a_{-1})$ are respectively a $\bt$- and $2\bt$-eigenvector for $\ad_{ a_0}$ and 
$$
 a_1=\frac{3\bt+1}{8\bt}  a_0+\frac{1}{\bt}\left [-\frac{3\bt+1}{8}  a_0+\frac{\bt}{2}( a_1+ a_{-1})+\frac{\bt}{2}( a_1- a_{-1})\right ],
$$
whence $\lmu=\frac{3\bt+1}{8\bt}=\frac{9\bt+1}{4}$. The Frobenius form is defined by $( a_i,  a_i)=1$ and $( a_i,  a_j)=\lmu$, for $i,j\in \{-1,0,1\}$ and $i\neq j$. 
 The projection graph (see~\cite{KMS} for the definition) has $ a_0$ and $ a_1$ as vertices and an edge between them since $( a_0,  a_1)\neq 0$. Thus it is connected and so by~\cite[Corollary~4.15 and Corollary~4.11]{KMS} 
every proper ideal of $V$  is contained in the radical of the form. Since the  determinant of the Gram matrix of the Frobenius form with respect to the basis $( a_{-1},  a_0,  a_1)$  is always non-zero, the algebra is simple.

If $ch\:\F= 3$, then condition $18\bt^2-\bt-1=0$ implies $\bt=2$. Hence $2\bt=1$ and $ a_0$ and $ a_1$ are both $1$-eigenvectors for $\ad_{ a_{0}}$. All the other properties are easily verified.
\end{proof} 

\begin{lemma}\label{algebra3A}
Let $\F$ be a field of characteristic other than $2$ and $\bt\in \F\setminus \{0,1, \frac{1}{2},  \frac{1}{4}\}$. The algebra $3A(2\beta, \beta)$ is a $2$-generated symmetric Frobenius axial algebra of Monster type $(2\bt,\bt)$ with  $\lmu=\lmd=\frac{\bt(9\bt-2)}{2(4\bt-1)}$. It is simple except when $(18\bt^2-\bt-1)(9\beta^2-10\beta+2)(5\beta-1)= 0$, in which case one of the following holds 
\begin{enumerate}
\item  $\beta=\frac{1}{5}$, $ch \:\F\neq 3$, and there is a unique quotient of maximal dimension which is isomorphic to $3C(\beta)$;
\item  $18\bt^2-\bt-1=0$, $ch \:\F\neq 3$, and there is a unique non trivial quotient  which is isomorphic to $V_3(\beta)$;
\item $9\beta^2-10\beta+2=0$, $ch \:\F\neq 3$, and there is a unique non trivial quotient  which is isomorphic to $1A$;
\item  $ch \:\F=3$, $\bt=-1$ and there are four non trivial quotients isomorphic respectively to $3C(-1), 3C(-1)^\times, V_3(-1), 1A$ (see~\cite[(3.4)]{HRS2} for the definition of $3C(-1)^\times$). 
 \end{enumerate}
\end{lemma}
\begin{proof}
Let $V$ be the algebra $3A(2\beta, \beta)$. Then $V$ has a Frobenius form and the projection graph (see~\cite{KMS} for the definition) has $ a_0$ and $ a_1$ as vertices and an edge between them since $( a_0,  a_1)\neq 0$. Thus it is connected and so by~\cite[Corollary~4.15 and Corollary~4.11]{KMS} 
every proper ideal of $V$  is contained in the radical of the form. The  Gram matrix of the Frobenius form with respect to the basis $( a_0,  a_1,  a_2,  s_{1,0})$ can be computed easily and 
has determinant 
$$
-\frac{\bt (9\beta^2-10\beta+2)^3(18\bt^2-\bt-1)(5\beta-1)}{16(4\bt-1)^5}.
$$
Suppose first $ch \:\F\neq 3$. When $\bt=\frac{1}{5}$ we see that the radical is generated by the vector $\frac{\bt}{2}( a_0+ a_1+ a_2)+ s_{0,1}  $ and hence the quotient over the radical is isomorphic to the algebra $3C(\bt)$. If $(18\bt^2-\bt-1)=0$, then the radical is generated by the vector $-\frac{\bt}{2}( a_0+ a_1+ a_2)+ s_{0,1}  $ and it follows that the quotient over the radical is isomorphic to the algebra $V_3(\bt)$, which by Lemma~\ref{algebraV3} is simple. Finally, if $(9\beta^2-10\beta+2)=0$, then the radical is three dimensional, with generators
$$
 a_0- a_2, \:\:  a_0- a_1, \:\: (2\bt-1) a_0 + s_{0,1}.
$$
It is immediate to see that the quotient over the radical is the trivial algebra $1A$. Using Lemma~\ref{ideal},  it is straightforward to prove that the radical is a minimal ideal. 

Now assume $ch \:\F= 3$. Then the radical of the form is three dimensional, with generators
$$
 a_0- a_2, \:\:  a_0- a_1, \:\:  s_{0,1}
$$
and it is straightforward to see that it contains properly the non-zero ideals $\F( a_0+ a_1+ a_2+ s_{0,1})$, $\F( a_0+ a_1+ a_2- s_{0,1})$, and $\F( a_0+ a_1+ a_2)$. Claim $(4)$ follows.
\end{proof}

\begin{lemma}\label{simple}
Let $\F$ be a field of characteristic other than $2$ and $\bt\in \F\setminus \{0,1, \frac{1}{2}\}$.
\begin{enumerate}
\item The algebra $3C(2\bt)$ has $ (\lmu, \lmd)=(\bt, 1)$. It is simple except when $\bt=-\frac{1}{2}$, in which case it has a non trivial quotient of dimension $2$, denoted by $3C(-\frac{1}{2})^\times$ (see~\cite[(3.4)]{HRS2}).
\item The algebra $3C(\bt)$ has $\lmu=\lmd=\frac{\bt}{2}$. It is simple except when $\beta=-1$, in which case it has a non trivial quotient of dimension $2$, denoted by $3C(-1)^\times$ (see~\cite[(3.4)]{HRS2}). 
\item The algebra $4J(2\bt,\bt)$ has $  (\lmu, \lmd)=(\bt, 0)$. It is simple except when $\beta= -\frac{1}{4}$, in which case it has a unique non trivial quotient over the ideal generated by $ a_{-1}+  a_0+  a_1+  a_2 -\frac{2}{\bt} s_{0,1}$, which is a simple algebra of dimension $4$.
\item The algebra $6J(2\bt,\bt)$  has $  (\lmu, \lmd)=\left (\bt, \frac{\bt}{2}\right )$. It  is simple provided 
$\beta \not \in \{ 2, -\frac{1}{7}\}$.
 \begin{quote}
  If $\bt=-\frac{1}{7}$, the algebra has a unique non trivial quotient, over the ideal $\F( a_{-2}+  a_{-1}+  a_0+  a_1+  a_2+  a_3 -\frac{1}{\bt} s_{0,3} -\frac{2}{\bt}  s_{0,1})$, which is a simple algebra of dimension $7$. \\
  If $\bt=2$, then it has a unique non trivial quotient  which is isomorphic to $3C(2\bt)$.
 
  \end{quote}
\end{enumerate}
\end{lemma}
\begin{proof}
(1) and (2) are proved in~\cite[(3.4)]{HRS2}.
Let $V\in \{4J(2\bt,\bt), 6J(2\bt,\bt)\}$. Then, $V$ is a subalgebra of a Matsuo algebra and so it is endowed of a Frobenius form. As in the proof of Lemma~\ref{algebra3A}, 
every proper ideal of $V$  is contained in the radical of the form. When $V=4J(2\bt,\bt)$, the Gram matrix, with respect to the basis $ a_{-1},  a_0,  a_1,  a_2, -\frac{2}{\bt} s_{0,1}$, is
$$
2\left ( 
\begin{array}{ccccc}
1&\beta &0 &\beta &2\beta\\
 \beta& 1 &\beta &0 &2\beta \\
  0&\beta &1&\beta &2\beta\\ 
 \beta & 0 & \beta & 1 & 2\beta\\
   2\beta&  2\beta& 2\beta & 2\beta & 2
\end{array}
\right ).
$$
The determinant of this matrix is $2(2\beta-1)^2(4\beta+1)$ and so, if $\bt\neq -\frac{1}{4}$ we get the thesis. If $\bt =-\frac{1}{4}$, the radical of the form is the $1$-dimensional ideal $\F( a_{-1}+  a_0+  a_1+  a_2 -\frac{2}{\bt} s_{0,1})$.
When $V=6J(2\bt,\bt)$, the Gram matrix, with respect to the basis $ a_{0},  a_{2},  a_{-2},  a_1,  a_{-1},  a_3, -\frac{1}{\bt} s_{0,3}, -\frac{2}{\bt}  s_{0,1}$ given in~\cite[Table~8]{DM}, is
$$
\left ( 
\begin{array}{cccccccc}
2&\beta &\beta&2\beta&2\beta &2\beta  &2\beta &4\beta\\
 \beta& 2 &\beta &2\beta&2\beta &2\beta  &2\beta &4\beta\\
  \beta&\beta &2&2\beta&2\beta &2\beta  &2\beta &4\beta\\ 
 2\beta&2\beta &2\beta  & 2& \beta & \beta & 2\beta& 4\beta\\
  2\beta&2\beta &2\beta  & \beta& 2 & \beta & 2\beta& 4\beta\\
   2\beta&2\beta &2\beta  & \beta& \beta & 2 & 2\beta& 4\beta\\
  2\beta&2\beta &2\beta  & 2\beta& 2\beta & 2\beta & 2& 2\beta\\
  4\beta&4\beta &4\beta  & 4\beta& 4\beta & 4\beta & 2\beta& 4+2\beta\\    
\end{array}
\right ).
$$
The determinant of this matrix is $-16(2\beta-1)^2(\beta-2)^5(7\beta+1)$ and so, if $\bt\not \in \{2,-\frac{1}{7}\}$, the algebra is simple. If $\bt=-\frac{1}{7}$, then the radical of the form is the $1$-dimensional ideal $\F( a_{-2}+  a_{-1}+  a_0+  a_1+  a_2+  a_3 -\frac{1}{\bt} s_{0,3} -\frac{2}{\bt}  s_{0,1})$ and the result follows. Finally suppose $\bt=2$. Then the radical of the form is $5$-dimensional with basis
$$
 a_0- a_2, \: a_0- a_{-2}, \: a_1- a_{-1},\:  a_1- a_3, \: s_{0,1}- s_{0,3}
$$
and the quotient over the radical is an algebra of type $3C(2\bt)$. Using Lemma~\ref{ideal},  it is straightforward to prove that the radical is a minimal ideal. 
\end{proof}

\begin{table}
{\Small
  $$
\begin{array}{c|c|c}
\mathrm{ Basis } & \mathrm{ Products } & \mathrm{ Form }\\
\hline
 a_0,   a_1,  a_2,  a_3,  s & \begin{array}{l}  a_i^2= a_i\\
 a_i a_{i+1}= s+\bt( a_i+ a_{i+1}) \:\:\:\:\:\:\:\:\:\:\:\:\:\mathrm{ where } \: a_4:= a_0\\
 a_i a_{i+2}=4\bt  s-\frac{2\bt-1}{2}( a_i+ a_{i+2})\: \mathrm{ where } \: a_5:= a_1\\
 a_i s=\bt  s+\frac{\bt^2}{2}( a_{i-1}+ a_{i+1}) \\
 s^2=\frac{3\bt-1}{8}(4 s- a_3 - a_{0}- a_{1} -  a_{2})
\end{array}&
\begin{array}{l}
( a_i,  a_i)=1\\
( a_i,  a_{i+1})=\bt +\frac{1}{4}\\
( a_i,  a_{i+2})=\bt\\
( a_i,  s)=\frac{1}{4}\bt\\
( s,  s)=\frac{1}{8}\bt
\end{array} \\
\hline
\end{array} 
$$
}
 \caption{Multiplication table for the algebra $4Y(2\bt,\bt)$}\label{tableY5} 
  \end{table}

\begin{lemma}\label{lemmaY5}
Let $\F$ be a field of characteristic other than $2$ and $\bt\in \F$ such that $4\bt^2+2\bt-1=0$. The algebra $4Y(2\bt,\bt)$ is a  simple $2$-generated primitive symmetric Frobenius axial algebra of Monster type $(2\bt,\bt)$, with  $\lmu=\bt+\frac{1}{4}$ and $\lmd=\bt$. 
\end{lemma}
\begin{proof}
All the properties are easily verified. Note that the Frobenius form is defined by $( a_i,  a_i)=1$, $( a_i,  a_j)=\lmu$, for $i,j\in \{0,1, 2,3\}$ such that $i-j\equiv_2 1$, $( a_0,  a_2)=( a_1,  a_3)=\lmd$, and $( a_i,  s)=\frac{1}{4}\bt$ for  $i\in \{0,1, 2,3\}$. Then the Frobenius is the one specified in Table~\ref{tableY5}, and its Gram matrix has determinant 
 $\frac{1}{32}\bt(\bt-1)^2(2\bt-1)(2\bt+3)$.  Note that, for $\bt=-\frac{3}{2}$, condition $4\bt^2+2\bt-1=0$ implies $ch \:\F=5$ and $\bt=1$. Hence the Frobenius form is always non-degenerate and  
 the result follows with the argument already used to prove Lemma~\ref{simple}.
\end{proof}

The next lemma is useful to deal the case of algebras of dimension at most $3$.
\begin{lemma}\label{nuovo1A}
Let $V$ be a symmetric primitive axial algebra of Monster type $(\al, \bt)$ over a field $\F$, generated by two axes $  a_0$ and $  a_1$. Suppose there exists $A\in \F$ such that 
$$
 a_2= a_{-1}+A( a_0- a_1).
$$
Then, one of the following holds 
\begin{enumerate}
\item $A=0$ and   $ a_2= a_{-1}$; 
\item $A=1$, $ a_1= a_{-1}$ and $V$ is spanned by $ a_0,  a_1,  s_{0,1}$. 
\end{enumerate}
\end{lemma}
\begin{proof}
If $A=0$, the claim is trivial. Suppose $A\neq 0$. By the symmetries of the algebra, we get 
$$
 a_{-2}= a_{1}+A( a_0- a_{-1}) \: \mbox{ and }  a_{3}= a_{0}+A( a_1- a_{2}).
$$
By substituting the expression for $ a_2$ in the definition of $ s_{0,2}$ we get
$$
0= s_{0,2}- s_{0,2}^{\tau_1}=A(1-2\bt)( a_0- a_2).
$$
Then, since $\bt\neq 1/2$, we have $ a_{2}= a_0$ and, by the symmetry, $ a_{-1}= a_1$. By Lemma~4.3 in~\cite{FMS1}, (2) holds.
\end{proof}

\begin{proposition}\label{jordan}
Let $V$ be a symmetric primitive axial algebra of Monster type $(2\bt, \bt)$ over a field $\F$ of characteristic other than $2$, generated by two axes $  a_0$ and $  a_1$. If $V$ has dimension at most $3$, then either $V$ is an algebra of Jordan type $\bt$ or $2\bt$, or $18\bt^2-\bt-1=0$ in $\F$ and $V$ is isomorphic to the algebra $V_3(\bt)$. 
\end{proposition}

\begin{proof}
Since $V$ is symmetric, $\ad_{ a_0}$ and $\ad_{ a_1}$ have the same eigenvalues. Since $1$ is an eigenvalue for $\ad_{ a_0}$, it follows from the fusion law that if $0$ is an eigenvalue for $\ad_{ a_0}$, or $V$ has dimension at most $2$,  then $V$ is of Jordan type $\bt$ or $2\bt$. Let us assume that $0$ is not an eigenvalue for $\ad_{ a_0}$. Then $ u_1=0$ (recall the definition of $u_1$ in Section~\ref{table}) and we get
$$
 s_{0,1}=[\lmu(1-2\bt)-\bt] a_0+\frac{\bt}{2}( a_1+ a_{-1}).
$$
Since we have also $ u_1^f=0$ we deduce
$$
 a_2= a_{-1}+\left [\frac{2}{\bt}(\lmu(1-2\bt)-\bt)-1\right ]( a_0- a_1).
$$
Thus we can apply Lemma~\ref{nuovo1A}. If claim (2) or (3) holds, then $V$ has dimension at most $2$ and we are done. Suppose claim (1) holds, that is $\frac{2}{\bt}(\lmu(1-2\bt)-\bt)-1=0$. Then 
$$
 s_{0,1}=\frac{\bt}{2}( a_0+ a_1+ a_{-1})\:\mbox{ and }  a_0 a_1=\frac{3}{2}\bt( a_0+ a_1)+\frac{\bt}{2} a_{-1},
$$
whence we get that $V$ satisfies the multiplication given in Table~\ref{tableV3} and so it is isomorphic to a quotient of $V_3(\bt)$. Since by hypothesis $\bt\not \in \{1,\frac{1}{2}\}$, by Lemma~\ref{algebraV3}, $V_3(\bt)$ is simple and $V\cong V_3(\bt)$. The vector $v:=3\bt a_0+(2\bt-1)( a_{-1}+ a_1)$ is a $2\bt$-eigenvector for $\ad_{ a_0}$ and, in order to satisfy the fusion law (in particular $v\cdot v$ must be a $1$-eigenvector for $\ad_{ a_0}$), $\bt$ must be such that $18\bt^2-\bt-1=0$.
\end{proof}

\begin{theorem}\label{thm2s} 
Let $V$ be a primitive symmetric axial algebra of Monster type $(2\bt, \bt)$ over a field $\F$ of characteristic other than $2$, generated by two axes $ a_0$ and $ a_1$.  Then, one of  the following holds:
\begin{enumerate}
\item $V$ is an algebra of Jordan type $\bt$ or $2\bt$; 
\item $18\bt^2-\bt-1=0$ in $\F$ and $V$ is isomorphic to $V_3(\bt)$;
\item $V$ is isomorphic to  $3A(2\bt, \bt)$;
\item  $V$ is isomorphic to  $4J(2\bt,\bt)$;
\item  $V$ is isomorphic to  $6J(2\bt,\bt)$; 
\item $4\bt^2+2\bt-1=0$ in $\F$ and $V$ is isomorphic to  $4Y(2\bt,\bt)$;
\item $\bt=-\frac{1}{4}$ and $V$ is isomorphic to the quotient of $4J(2\bt,\bt)$ over the one-dimensional ideal generated by $ a_{-1}+ a_{0}+ a_{1}+ a_2-\frac{2}{\bt} s_{0,1}$;
\item  $\bt=-\frac{1}{7}$ and $V$ is isomorphic to the quotient of $6J(2\bt,\bt)$ over the one-dimensional ideal generated by  $ a_{-2}+ a_{-1}+ a_{0}+ a_{1}+ a_2+ a_{3}-\frac{2}{\bt} s_{0,1}-\frac{1}{\bt} s_{0,3}$.
\end{enumerate}
\end{theorem}
\begin{proof}
  By Theorem~\ref{nec}, $V$ is determined, up to homomorphic images, by the pair $ ( \lambda_1,  \lambda_2)$, which must be a solution of~(\ref{syst}).  By~\cite[Corollary~3.8 ]{FMS1} and Proposition~\ref{span}, $V$ is spanned on $F$ by the set $ a_{-2}$, $ a_{-1}$, $ a_{0}$, $ a_{1}$,$ a_{2}$,$ a_{3}$, $ s_{0,1}$,  and $ s_{0,3}$. 

 Assume first $ \lambda_1= \bt$. Then, by Lemma~\ref{s3f}, we get $ s_{1,3}= s_{2,3}= s_{0,3}$. If $( \lambda_1,  \lambda_2)=(\bt, \frac{\bt}{2})$, we see that the algebra satisfies the multiplication table of the algebra $6J(2\bt,\bt)$. Hence $V$ is isomorphic to a quotient of $6J(2\bt,\bt)$ and by Lemma~\ref{simple} we get that either (5) or (8) holds. Assume $\lmd\neq \frac{\bt}{2}$. We  compute 
 $$
 E=\frac{(2\bt-1)(2\lmd-\bt)}{4}\left [  s_{0,3}+\bt ( s_{0,1}- a_{-1}+ a_3)\right ],
$$
hence, since $(2\bt-1)(2\lmd-\bt)\neq 0$,  we get  $ s_{0,3}=\bt( a_{-1}- a_3- s_{0,1})$. 
Then, from the identity $ s_{0,3}- s_{0,3}^{\tau_1}=0$ we get $ a_3= a_{-1}$, and so $ s_{0,1}= s_{0,3}$ and $ a_{-2}= a_2$. Hence the dimension is at most $5$. If $( \lambda_1,  \lambda_2)= (\bt, 0 )$ we see that $V$ satisfies the multiplication table of $4J(2\bt,\bt)$ and either (4) or (7) holds. Finally, if $( \lambda_1,   \lambda_2)= (\bt, 1 )$, then $Z(\bt, \bt)=0$ and so by Lemma~\ref{a0s2f} and Equation~(\ref{s}) we get $ a_0 a_2= a_0$, that is $ a_0$ is a $1$-eigenvector for $\ad_{ a_2}$. By primitivity, this implies $ a_2= a_0$. Consequently, we have $ a_{-1}= a_2^f= a_0^f= a_1$ and from the multiplication table we see that $V$ is a quotient of the algebra $3C(2\bt)$. Thus $V$ an axial algebra of Jordan type $2\bt$.

Now assume $\lmu\neq \bt$. We have
\begin{eqnarray*}
\lefteqn{ D_1=\frac{-2(2\bt-1)(\lmu-\bt)}{\bt^2}\left [(\bt-1) \frac{}{}( a_{-2}- a_2) \right .}\\
& & \:\:\:\:\:\:\:\:\:\:\:\:\:\:\:\:\:\:\:\:\:\:\:\:\:\:\:\:\:\:\:\:\:\:\:\:\:\:\left .+\left (\frac{2\lmu(4\bt-1)(2\lmu-3\bt)}{\bt^2}+10\bt-3-2\lmd \right )( a_{-1}- a_1)\right ].
\end{eqnarray*}
By Lemma~\ref{u1u2}, $ D_1=0$. Since $\lmu\neq \bt$ and $\bt\not \in \{1,\frac{1}{2}\}$, the coefficient of $ a_{-2}$ in $ D_1$ is non zero and we get
\begin{equation}\label{am2}
 a_{-2}= a_2+\frac{1}{(\bt-1)}\left [\frac{2\lmu(4\bt-1)(2\lmu-3\bt)}{\bt^2}+10\bt-3-2\lmd \right ]( a_{1}- a_{-1}).
\end{equation}
Since $V$ is symmetric, the map $f$ swapping $ a_0$ and $ a_1$ is an algebra automorphism and so 
\begin{equation}\label{A3}
 a_{3}= a_{-1}+\frac{1}{(\bt-1)}\left [\frac{2\lmu(4\bt-1)(2\lmu-3\bt)}{\bt^2}+10\bt-3-2\lmd \right ]( a_{0}- a_2).
\end{equation}
It follows also $ s_{0,3}\in \langle  a_{-1},  a_0,  a_1,  a_2,  s_{0,1}\rangle$ and hence $V=\langle  a_{-1},  a_0,  a_1,  a_2,  s_{0,1}\rangle$. Moreover, equation $ d_1=0$ of Lemma~\ref{u1u2} becomes
\begin{equation}\label{diff1}
B( a_{-1}- a_2)+C( a_0- a_1)=0
\end{equation}
with $B$ and $C$ in $\F$.

Assume $\bt\neq \frac{1}{4}$ and $( \lambda_1, \lambda_2)=\left (\frac{(9\bt^2-2\bt)}{2(4\bt-1)}, \frac{(9\bt^2-2\bt)}{2(4\bt-1)}\right )$ (note that $\lmu\neq \bt$ since  $\bt\neq 0$). Then
the identities $ a_{-2}^2= a_{-2}$ and $ a_{3}^2= a_{3}$ give the equations
$$
\frac{2\bt(18\bt-5)}{(4\bt-1)}( a_2- a_{-1})=0 \mbox{ and } \frac{2(18\bt^2-9\bt+1)}{(4\bt-1)}( a_2- a_{-1})=0 .
$$
Since, in any field $\F$, the two polynomials $(18\bt-5)$ and $(18\bt^2-9\bt+1)$ have no common roots, we have $ a_2= a_{-1}$, whence $ a_{-2}= a_{1}$, and it is straightforward to see that $V$ satisfies the multiplication table of the algebra $3A(2\bt, \bt)$. Hence the result follows from Lemma~\ref{algebra3A}.

Now assume $( \lambda_1,  \lambda_2)=\left(\frac{\bt}{2}, \frac{\bt}{ 2}\right )$. Then we have 
$$
B=\frac{(\bt-1)(4\bt-1)}{2\bt} \mbox{ and } C=0.
$$
Moreover, the identity $ a_{-2}^2= a_{-2}$ give the equation
$$
\frac{(\bt^2+2\bt-1)}{\bt}( a_2- a_{-1})=0
$$
Suppose $B\neq 0$, or $(\bt^2+2\bt-1)\neq 0$: we get $ a_{-1}= a_2$ and consequently $ s_{0,2}= s_{0,1}$.  From the identity $ s_{0,2}- s_{0,1}=0$ we get 
$$
\frac{(5\bt-1)}{\bt}\left [ s_{0,1}+\frac{\bt}{2}( a_0+ a_1+ a_{-1})\right ]=0.
$$  
If $\bt\neq \frac{1}{5}$, it follows that the dimension is at most $3$ and in fact we get a quotient of the algebra $3C(\bt)$ which is an algebra of Jordan type $\bt$. If $\bt=1/5$, we have $\frac{\bt}{2}=\frac{(9\bt^2-2\bt)}{2(4\bt-1)}$ and $V$ is a quotient of the algebra $3A(2\bt,\bt)$. Finally, if $B=0=(\bt^2+2\bt-1)$, then we get $ch\: \F=7$ and $\bt=2$. In this case, by Equation~(\ref{am2}), we have $ a_{-2}= a_2+ a_1- a_{-1}$. Computing the vectors $ u_2$ and $ v_2$ with Lemma~\ref{lambdaa}, we get $ v_2=0$ and hence $0= a_2- \lambda_2  a_0-u_2-w_2=2( a_1- a_{-1})$. Thus, $ a_{-1}= a_1$ and $ a_2= a_0$, $V$ has dimension at most $3$ and the result  follows from Proposition~\ref{jordan}.

Assume  $( \lambda_1, \lambda_2)= (0,1)$. Then 
$$
B=C=\frac{-(9\bt-4)(4\bt^2+2\bt-1)}{\bt(\bt-1)},
$$ 
the identity  $ a_{3}^2= a_{3}$ gives the equation
\begin{equation}\label{2B}
\frac{5(2\bt-1)}{\bt(\bt-1)^2}\left [-2B( a_1- a_{-1})+(18\bt^3+27\bt^2-24\bt+4)( a_0- a_{2})\right ]=0, 
\end{equation}
and the identity $ D_2=0$ of Lemma~\ref{u1u2} gives 
$$
(4\bt-1)B( a_{-1}- a_1)=0
$$
Thus, if $(4\bt-1)B\neq 0$, we get 
$ a_{-1}= a_1$. Then,  $ a_2= a_{-1}^f= a_1^f= a_0$ and hence $ s_{0,2}=(1-2\bt) a_0$. Further, from the first equation of Lemma~\ref{a0s2f}, we also get $ s_{0,1}=-\bt( a_0+ a_1)$, whence $ a_0  a_1=0$. Thus $V$ is isomorphic to the algebra $2B$.
Suppose $B=0$. If  $(18\bt^3+27\bt^2-24\bt+4)\neq 0$, from Equation~(\ref{2B}), we get $ a_0= a_{2}$ and, by the symmetry, also $ a_1= a_{-1}$. Thus the dimension is at most $3$ and we conclude by Proposition~\ref{jordan}. If $(18\bt^3+27\bt^2-24\bt+4)= 0$, then it follows $ch \:\F=31$ and $\bt=9$. In this case $ E- E^{\tau_0}\neq 0$: a contradiction to Lemma~\ref{u1u2}.  Now assume $\bt=\frac{1}{4}$ and $B\neq 0$ (thus, in particular, $ch \:\F\neq 3$). From Equation~(\ref{am2}) we get $ a_{-2}= a_2+\frac{10}{3}( a_1- a_{-1})$. Further, $ v_2=0$ and we get $0= a_2- \lambda_2  a_0-u_2-w_2=\frac{10}{3}( a_1- a_{-1})$. Thus, if $ch \:\F\neq 5$, we get $ a_{-1}= a_1$ and $ a_2= a_0$; $V$ has dimension at most $3$ and the result  follows from Proposition~\ref{jordan}. If $ch \:\F=5$, then $B=-1$, Equation~(\ref{diff1}) gives $ a_{2}= a_{-1}-( a_0- a_1)$ and so Lemma~\ref{nuovo1A} implies that $V$ has dimension at most $3$ and we conclude as above.

Suppose $( \lambda_1, \lambda_2)=(1,1)$. Then
identity $ D_2=0$ is equal to
$$
\frac{2(4\bt^3-14\bt^2+11\bt-2)(9\bt^2-10\bt+2)(\bt-2)(2\bt-1)}{\bt^6}( a_{-1}- a_1)=0
  $$
Hence, if $(4\bt^3-14\bt^2+11\bt-2)(9\bt^2-10\bt+2)(\bt-2)\neq 0$, we get $ a_{-1}= a_1$,  by the symmetry, $ a_2= a_0$  and $V$ has dimension at most $3$ and we conclude by Proposition~\ref{jordan}. So let us assume 
\begin{equation}\label{condition}
(4\bt^3-14\bt^2+11\bt-2)(9\bt^2-10\bt+2)(\bt-2)= 0.
\end{equation}
Further,  we have 
$$
B=-\frac{1}{\bt^3}\left (36\bt^4-126\bt^3+127\bt^2-48\bt+6\right ) 
$$
and 
$$
C=-\frac{1}{\bt^4}(4\bt^2+2\bt-1)(9\bt^2-10\bt+2)(\bt-2).
$$
If $(4\bt^3-14\bt^2+11\bt-2)=0$, then $B$ and $C$ are not zero and by Lemma~\ref{nuovo1A}, $V$ has dimension at most $3$ and we conclude by Proposition~\ref{jordan}. Moreover, if $\bt=2$ then $(\lmu, \lmd)=(\frac{\bt}{2},\frac{\bt}{2})$ and we reduce to the previous case. Hence we assume 
$$
(9\bt^2-10\bt+2)= 0.
$$
Then $C=0$. If $B\neq 0$, from Equation~(\ref{diff1}) we get $ a_2= a_{-1}$ and by symmetry, $ a_1= a_{-2}$. Then, Equation~(\ref{am2}) becomes
$$
 a_{1}- a_{-1}=\frac{1}{(\bt-1)}\left [\frac{2(4\bt-1)(2-3\bt)}{\bt^2}+10\bt-5 \right ]( a_{-1}- a_1),
$$
whence, either $ a_1= a_{-1}$ and again $V$ has dimension at most $3$ and we conclude by Proposition~\ref{jordan}, or 
$$
\frac{1}{(\bt-1)}\left [\frac{2(4\bt-1)(2-3\bt)}{\bt^2}+10\bt-5 \right ]=-1
$$
that is
$$
\frac{(11\bt^3-30\bt^2+22\bt-4)}{\bt^2(\bt-1)}=0.
$$
It is now straightforward to check that the two polynomials $(9\bt^2-10\bt+2)$ and $(11\bt^3-30\bt^2+22\bt-4)$ have no common roots in any field of characteristic other than 2 and we get a contradiction.
We are now left to consider the case when $B=0$. Then the two polynomials $(9\bt^2-10\bt+2)$ and $B$ have a common root if and only if $ch \:\F\in \{3,7\}$ and the common root if $\bt=-1$. In both cases we get $(\lmu, \lmd)=(1,1)=\frac{(9\bt^2-2\bt)}{2(4\bt-1)}$ and so $V$ is a quotient of the algebra $3A(2\bt,\bt)$ as already proved.

Suppose now $\bt=\frac{1}{4}$ and $(\lmu, \lmd)=(\mu, 1)$, for $\mu\in \F\setminus \{0,1, \bt\}$. In particular, note that $ch \:\F\neq 3$, since $\bt\neq 1$. Then Equations~(\ref{am2}) and (\ref{A3}) become 
$$
 a_{-2}= a_2+\frac{10}{3}( a_{-1}- a_1) \: \mbox{ and }  a_{3}= a_{-1}+\frac{10}{3}( a_{2}- a_0)
$$
and 
 the identity $ D_2=0$ gives the relation
$$
\frac{28}{3}(4\mu-1)( a_{1}- a_{-1})=0.
$$
Since we are assuming $\lmu\neq \bt$, $(4\mu-1)\neq 0$. So, if $ch \:\F\neq 7$,  we get $ a_{1}= a_{-1}$. Hence $V$ ha dimension at most $3$ and we conclude by Proposition~\ref{jordan}. If $ch \:\F=7$, then $\bt=2$ and we conclude, with the same argument used above for the case when $(\lmu, \lmd)=(\frac{\bt}{2}, \frac{\bt}{2})$, $ch \:\F=7$,  $\bt=2$. 

Finally assume that $\bt\neq \frac{3}{8}$, $\bt\in \{\frac{1}{2}, \frac{1}{3}, \frac{2}{7}\}$ or $\bt$
satisfies Equation~(\ref{condition}), and $(\lmu, \lmd)=\left (\frac{(18\bt^2-\bt-2)}{2(8\bt-3)}, \frac{(48\bt^4-28\bt^3+7\bt-2)(3\bt-1)}{2\bt^2(8\bt-3)^2}\right )$. First of all, note that the only common solution of Equation~(\ref{condition}) and of the equation $2\bt^2+5\bt-2=0$ is $\bt=1$ when $ch \:\F=5$. Since we are assuming $\bt\neq 1$, under our hypotheses  $2\bt^2+5\bt-2$ is always non-zero in $\F$: in particular this implies $\lmu\neq \bt$. The identity $ D_2=0$ becomes
$$
\frac{(2\bt^2+5\bt-2)(4\bt^4+2\bt-1)(5\bt^2+\bt-1)(7\bt-2)(2\bt-1)}{\bt^5(8\bt-3)^2(\bt-1)}( a_1- a_{-1})=0.
$$
If $(4\bt^4+2\bt-1)(5\bt^2+\bt-1)(7\bt-2)\neq 0$ in $\F$, we deduce $ a_{1}= a_{-1}$ and, by symmetry, $ a_{0}= a_{2}$. Thus $V$ has dimension at most $3$ and we conclude by Proposition~\ref{jordan}.
If $(5\bt^2+\bt-1)=0$, then $\lmu=\lmd=\frac{\bt}{2}$ and we are in a case considered above.  If $(4\bt^4+2\bt-1)=0$, we get $\lmu=\bt+\frac{1}{4}$ and $\lmd=\bt$. From the identity $ E=0$ we get $ a_3= a_{-1}$, consequently $ a_{-2}= a_2$, and it follows that $V$ satisfies the multiplication table of the algebra $4Y(2\bt,\bt)$. Hence, by Lemma~\ref{lemmaY5},   (6)  holds. Finally assume $ch \:\F\neq 7$ and $\bt=\frac{2}{7}$. Then $\lmu=\lmd=\frac{4}{7}$ and, since $\bt\neq 1$, we have also $ch \:\F\neq 5$. From identity $ d_1=0$, we get  $ a_{-2}= a_2$ and consequently $ a_3= a_{-1}$. If further $ch \:\F\neq 3$, identity $ E=0$ implies $ a_{-1}= a_2$. Then $ a_0= a_1$, but $\lmu=\frac{4}{7}\neq 1$, a contradiction. Hence $ch \:\F=3$, $(\lmu, \lmd)=(\frac{\bt}{2},\frac{\bt}{2})$ and we conclude as in the case considered above. \end{proof}

\section{The non symmetric case}\label{ns}
In this section we deal with the nonsymmetric case and prove Theorem~\ref{thm2}. 
Let $V$ be generated by the two axes $ a_0$ and $ a_1$. We set  $V_e:=\langle \langle  a_0,  a_2\rangle \rangle$ and $V_o:=\langle \langle  a_{-1},  a_1\rangle \rangle$. As noticed at the end of Section~\ref{strategy}, $V_e$ and $V_o$ are symmetric, since the automorphisms  $\tau_1$ and $\tau_0$ respectively, swap their generating axes. Hence, from Theorem~\ref{thm2s} we get the possible values for the pair $( \lmd, \lmdf)$ and the structure of those subalgebras. Note that $V$ is symmetric if and only if $\lm=\lmf$ and $\lmd=\lmdf$ in $\F$. 

\begin{lemma}\label{8dim}
If $V$ has dimension $8$, then $V\cong 6J(2\bt,\bt)$.
\end{lemma}
\begin{proof}
Suppose $V$ has dimension $8$. Then the generators $ a_{-2},  a_{-1},  a_0,  a_1,  a_2,  a_3,  s_{0,1}$, and $ s_{0,3}$ are linearly independent. We express $ d_1$ defined before Lemma~\ref{u1u2} as a linear combination of the basis vectors and since $ d_1=0$ in $V$, every coefficient must be zero.
In particular, considering the coefficients of $ a_{-2},  a_3$, and $ s_{0,1}$ we get, respectively, the equations
\begin{eqnarray*}
&&\frac{(6\bt-1)}{\bt}\lmu +\frac{(2\bt-2)}{\bt}\lmu^f-(8\bt-3)=0\\
&&\frac{(2\bt-2)}{\bt}\lmu +\frac{(6\bt-1)}{\bt}\lmu^f-(8\bt-3)=0\\
&& \frac{8}{\bt^2}(\lmu^2-{\lmf}^2)-\frac{24}{\bt}(\lmu-\lmf)-\frac{2}{\bt}(\lmd-\lmdf)=0
\end{eqnarray*}
whose common solutions have $\lmu=\lmf$ and $\lmd=\lmdf$. Hence $V$ is symmetric and the result follows from Theorem~\ref{thm2s}.  
\end{proof}

\begin{lemma}\label{evenV8(-1/7)}
If $V$ is non-symmetric, then $V_e$ and $V_o$ are not isomorphic to the seven dimensional quotient of $6J(-\tfrac{2}{7}, -\tfrac{1}{7})$.
\end{lemma}
\begin{proof}
It is enough to show that the claim holds for $V_o$. Let us assume by contradiction that 
$V_o$ is isomorphic to seven dimensional quotient of $6J(-\tfrac{2}{7}, -\tfrac{1}{7})$. By Lemma~\ref{8dim}, $V$ has dimension smaller than $8$, hence $V=V_e$ and hence has basis $a_{-5}, a_{-3}, a_{-1}, a_1, a_3, a_5, s_{1,2}$ and the product  with respect to this basis is known (see~\cite[Table~8]{DM}). Let $a_0=x_{-5}a_{-5} +x_{-3}a_{-3}+x_{-1}a_{-1}+ x_{1}a_{1} +x_{3}a_{3}+x_{5}a_{5}+zs_{1,2}$ be the decomposition of $a_0$ with respect to this basis. Since $a_0^{\tau_0}=a_0$, and $\tau_0$ swaps $a_{-i}$ and $a_i$, for $i\in \{1,3,5\}$ and fixes $s_{1,2}$, must be $x_{-i}=x_i$, for $i\in \{1,3,5\}$. Moreover, $a_1-a_{-1}$ and $a_3-a_{-3}$ are $\bt$-eigenvectors for $a_0$, and so we have
\begin{eqnarray*}
\lefteqn{0=14a_0(a_1-a_{-1})-\bt(a_1-a_{-1})=}\\
&=& (x_3+x_5-2z)(a_5-a_{-5})+(x_3+x_5)(a_{-3}-a_3)\\
&&+ (14x_1-3x_3-3x_5-2z+2)(a_1-a_{-1})
  \end{eqnarray*}        
and, since all the coefficients must be zero, we get  $x_5=-x_3$, $z=0$,  and $x_1=-\tfrac{1}{7}$. Then
\begin{eqnarray*}
\lefteqn{0=14a_0(a_3-a_{-3})-\bt(a_3-a_{-3})=}\\
&=&(x_3+\tfrac{1}{7})(a_{1}-a_{-1}+a_5-a_{-5})- 17(x_3+\tfrac{1}{7})(a_3-a_{-3})
 \end{eqnarray*} 
 whence $x_3=-\frac{1}{7}$ and $a_0=\frac{1}{7}(a_{-5}-a_{-3}-a_{-1}-a_1-a_3+a_5)$. But then $a_0^2\neq a_0$, a contradiction.
\end{proof}

Lemma~\ref{8dim}, Lemma~\ref{evenV8(-1/7)} and Theorem~\ref{thm2s} imply that if $V$ is non-symmetric, then the dimensions of the even subalgebra and of the odd subalgebra are both at most $5$. As a consequence, from Lemma~\ref{a3} we derive some relations between the odd and even subalgebras.

\begin{lemma}\label{a4=am4}
If $ a_{-3}= a_5$, then 
\begin{eqnarray}\label{ast1}
\lefteqn{ Z(\lmf, \lmu)( a_0- a_{2}+ a_{-2}- a_4)=}\\
&=&\frac{1}{\bt}\left [2Z(\lmf, \lmu)\left (\lmf-\bt\right )-\left (\lmdf-\bt\right )\right ]( a_{-1}- a_{3}). \nonumber 
\end{eqnarray}
If  $ a_{-4}= a_4$, then 
\begin{eqnarray}\label{ast2}
\lefteqn{Z(\lmu, \lmf)( a_{-1}- a_3+ a_{-3}- a_1)=}\\
&=&\frac{1}{\bt}\left [2Z(\lmu, \lmf)\left (\lmu-\bt\right )-\left (\lmd-\bt\right )\right ]( a_{-2}- a_2). \nonumber 
\end{eqnarray}
If $ a_{-3}= a_3$, then 
\begin{eqnarray}\label{ast3}
\lefteqn{2Z(\lmf, \lmu)( a_{2}- a_{-2})=}\\
&=&\frac{1}{\bt}\left [4Z(\lmf, \lmu)\left (\lmf-\bt\right )-\left (\lmdf-\bt\right )\right ]( a_{1}- a_{-1}). \nonumber 
\end{eqnarray}
If $ a_{-2}= a_4$, then 
\begin{eqnarray}\label{ast4}
\lefteqn{2Z(\lmu, \lmf)( a_{-1}- a_{3})=}\\
&=&\frac{1}{\bt}\left [4Z(\lmu, \lmf)\left (\lmu-\bt\right )-\left (\lmd-\bt\right )\right ]( a_{0}- a_{2}). \nonumber 
\end{eqnarray}
If $ a_{-2}= a_4$ and $ a_3= a_{-1}$, then 
\begin{eqnarray}\label{ast5}
\lefteqn{2\left [\bt Z(\lmu, \lmf)-2(\lmf-\bt)\right ]( a_{2}- a_{-2})=}\\
&=&\frac{1}{\bt^2}\left [4(2\bt-1)\bt Z(\lmu, \lmf)(\lmf-\bt)-8\bt(\lmu-\lmf)\left ( 2\bt-\lmu-\lmf \right ) \right . \nonumber \\
&&\left .+2(2\bt-1)^2(\lmf-\bt)-2\bt^2 (\lmd-\lmdf)\right ]( a_{1}- a_{-1}). \nonumber 
\end{eqnarray}
\end{lemma}
\begin{proof}
By applying the maps $\tau_0$ and $\tau_1$ to the formulas of Lemma~\ref{a3}  we find similar formulas for $a_{-4}$ and $a_5$. Equations~(\ref{ast1}),  (\ref{ast2}), (\ref{ast3}), and (\ref{ast4}) follow. To prove Equation~\ref{ast5}, note that if $ a_{-2}= a_4$ and $ a_3= a_{-1}$, then $ s_{0,3}= s_{0,1}$. Thus $ s_{1,3}- s_{2,3}=0$ and the claim follows from Lemma~\ref{s3f}.\end{proof}

In view of the above relations, it is important to investigate some subalgebras of the symmetric algebras.
\begin{lemma}\label{subalgebras}
$\:$
\begin{enumerate} 
\item If $U$ is one of the algebras   $4Y(2\bt,\bt)$, $4J(2\bt,\bt)$, or its four dimensional quotient,  then, 
$ U=\langle \langle  a_{-1}- a_1,  a_{0}- a_2\rangle \rangle.
$ 
\item If $U$ is one of the algebras  $3C(\bt)$,  $3C(-1)^\times$, $3A(2\bt,\bt)$, and $V_3(\bt)$,  then, 
$ U=\langle \langle  a_{-1}- a_1,  a_{0}- a_1,\rangle \rangle,
$
unless $U=3C(2)$ and $ch \:\F=5$, or $ch \:\F=3$ and $V=3C(-1)^\times$ or $V=V_3(-1)$.
 \end{enumerate}
\end{lemma}
\begin{proof}
To prove (1), set $W:=\langle \langle  a_{-1}- a_1,  a_{0}- a_2\rangle \rangle$. Let $U$ be equal to $4J(2\bt,\bt)$ or its four dimensional quotient when $\bt=-\frac{1}{4}$. Then $ a_{-1} a_1=0= a_0  a_2$. Thus $( a_{-1}- a_1)^2= a_{-1}+ a_{1}$ and $( a_{0}- a_2)^2= a_{0}+ a_{2}$, whence we get that  $ a_0,  a_1$ belong to $W$ and the claim follows.

Let $U$ be equal to $4Y(2\bt,\bt)$. Then,  $W$ contains the vectors
\begin{eqnarray*}
&&( a_{-1}- a_1)^2= a_{-1}+ a_{1}+(2\bt-1)( a_0+ a_2)-8\bt  s_{0,1},\\
&&( a_{0}- a_2)^2= a_{0}+ a_{2}+(2\bt-1)( a_{-1}+ a_1)-8\bt  s_{0,1},\\
&& a_2- a_1=\frac{1}{4(\bt-1)}\left \{( a_{-1}- a_1)^2-( a_{0}- a_2)^2-2(\bt-1)[( a_{-1}- a_1)+(  a_{0}- a_2)]\right \},
\end{eqnarray*}
and
$$
( a_{2}- a_1)^2= a_2+ a_1-2\bt( a_2+ a_1)-2 s_{0,1}.
$$
It is straightforward to check that the five vectors $ a_{-1}- a_1$, $ a_{0}- a_2$, $( a_{-1}- a_1)^2$, $ a_2- a_1$, and $( a_{2}- a_1)^2$ generate the entire algebra $4Y(2\bt,\bt)$. 

To prove (2), set $W:=\langle \langle  a_{-1}- a_1,  a_{0}- a_1,\rangle \rangle$. Let $U$ be equal to $3C(\bt)$. Then, $W$ contains also $( a_{0}- a_1)^2=(1-\bt)( a_0+ a_{1})+\bt a_{-1}$ and the three vectors generate $U$, unless $\bt=2$ and $ch\:\F=5$. If $V=3A(2\bt, \bt)$, then $W$ contains also the vectors $( a_0- a_1)^2=(1-2\bt)( a_0+ a_1)-2 s_{0,1}$ and $( a_0- a_1)( a_{-1}- a_1)=(1-2\bt) a_1 - s_{0,1}$. Thus we get four vectors that generate $U$.
If $U=V_3(\bt)$, then $W$ contains $( a_0- a_1)^2=(1-3\bt)( a_0+  a_1)-\bt a_{-1}$ and we get that $W$ contains three vectors that generate $U$, unless $\bt=2$ and $ch\:\F=3$. Finally, if $U=3C(-1)^\times$, then  $ a_0- a_1$ and $( a_0- a_1)^2=3( a_0+ a_1)$ generate $U$, unless $ch \:\F\neq 3$.
\end{proof}

We start now to consider the possible configurations.  Many of them are forbidden by the following result which is a immediate consequence of Theorem~1.1 in~\cite{forbidden}. 

\begin{proposition}\label{forbidden}
Let $V$ be a primitive $2$-generated axial algebra of Monster type with generators $ a_0$ and $ a_1$. Let $D_e:=|\{ a_i\:|\: i \in 2\Z\}|$ and $D_o:=|\{ a_i\:|\: i \in 1+2\Z\}|$. Then one of the following occur: $D_e=D_o$, $D_e=2D_o$, or $2D_e=D_o$.
\end{proposition}

\begin{lemma}\label{evenV5}
If $V$ is non-symmetric and $D_e=4$, then $V=V_o$.
\end{lemma}
\begin{proof}
Let us assume that 
$D_e=4$. Then, by Lemma~\ref{forbidden}, $D_o\in \{2,4\}$. Moreover, the vectors $ a_{-2},  a_0,  a_2,  a_4$ are linearly independent and, from the first formula in Lemma~\ref{a3}, we get that $Z(\lmu, \lmf)\neq 0$ and  $ a_{3}\neq  a_{-1}$. Hence $D_o=4$ and $a_1$, $a_{-1}$, $a_{3}$, $a_{-3}$ are linearly independent. Since $a_{-4}= a_4$, Equation~(\ref{ast2}) holds. 

Suppose $2Z(\lmu, \lmf)\left (\lmu-\bt\right )-\left (\lmd-\bt\right )= 0$.  Then, by Equation~(\ref{ast2}), we have $ a_1- a_{-1}+ a_{3}- a_{-3}=0$, a contradiction. 
Hence $2Z(\lmu, \lmf)\left (\lmu-\bt\right )-\left (\lmd-\bt\right )\neq 0$. Then $( a_{-2}- a_2)\in \langle \langle  a_{-1},  a_1\rangle \rangle$. Since $V_o$ is invariant under $\tau_1$, it contains also $ a_{0}- a_4$. Thus by Lemma~\ref{subalgebras}, $V_e\subseteq V_o=V$.  
\end{proof}

\begin{lemma}\label{even33}
If $V$ is non-symmetric and $D_e=3$, then either $V=V_e$ or $V=V_o$.
\end{lemma}

\begin{proof}
Let us assume that $D_e=3$, that is
$ a_4= a_{-2}$. Then, by Proposition~\ref{forbidden}, $D_o=3$ as well. In particular, $ a_{-3}= a_3$ and Equations~(\ref{ast3}) and~(\ref{ast4}) hold. 

If $Z(\lmu, \lmf)=0=Z(\lmf, \lmu)$, then it follows $\lmu=\lmf$ and,  by  Equations~(\ref{ast3}) and~(\ref{ast4}),  $\lmd=\lmdf=\frac{\bt}{2}$. Thus $V$ is symmetric: a contradiction. Therefore, without loss of generality, we may assume $Z(\lmf, \lmu)\neq 0$.
 Then, Equation~(\ref{ast3}) implies 
$$ a_2- a_{-2},\:\:  a_0- a_2 \in V_o.
$$
 If $V_e=\langle \langle  a_2- a_{-2},\:\:  a_0- a_2 \rangle \rangle$, then $V_e\subseteq V_o=V$ and we are done. Otherwise (since $\bt\neq \frac{1}{2}$)  Lemma~\ref{subalgebras} yields that 
  $$ch \:\F=5, \:\:\bt=2, \:\mbox{ and }\:\:V_e\cong 3C(2).
  $$
   If also $Z(\lmu, \lmf)\neq 0$ or $4Z(\lmf, \lmu)\left (\lmf-\bt\right )-\left (2\lmdf-\bt\right )\neq 0$,  from Equations~(\ref{ast4}) and~(\ref{ast3}), respectively, we get that $ a_{-1}- a_{3}$ and $  a_{-1}- a_1$ are contained in $V_e$ and so,  as above, by Lemma~\ref{subalgebras}, either $V_e=V$ or $V_o\cong 3C(2)$. In the latter case, from the formulas in Lemma~\ref{a0s2f} we find
  $$
   s_{0,1}=(\lmu-\bt) a_0-\bt( a_{-1}+ a_1) \:\mbox{ and } \: s_{0,1}=(\lmf-\bt) a_1-\bt( a_{0}+ a_2). 
  $$
Comparing the two expressions and using the invariance of $ s_{0,1}$ under $\tau_0$ and $\tau_1$,  we get
\begin{eqnarray*}
\lmu a_0-\bt a_{-1}&=& \lmf a_1-\bt a_2\\
(\lmu-\bt)( a_0- a_2)&=&\bt( a_{-1}- a_3)\\
(\lmf-\bt)( a_1- a_{-1})&=&\bt( a_{2}- a_{-2}).
\end{eqnarray*}
From the above identities we can express $ a_3,  a_{-2}, $ and $ a_{-1}$ as linear combinations of $ a_0,  a_1,  a_2$. So $V$ has dimension at most $3$ and  hence $V= V_o=V_e$.

Suppose finally $Z(\lmu, \lmf)=0 \mbox{ and } \:4Z(\lmf, \lmu)\left (\lmf-\bt\right )-\left (2\lmdf-\bt\right )= 0.$ 
 Then, we have $\lmf=2\lmu+3$,  $\lmdf=\lmu(\lmu+1)$, and    
 $$
 \begin{array}{l}
 p_2(\lmu,  2\lmu+3, 1, \lmu(\lmu+1))=\lmu(\lmu-2)\\
 p_4(\lmu,  2\lmu+3, 1, \lmu(\lmu+1))=2\lmu^2-\lmu-1.
 \end{array}
 $$
  Hence, by Theorem~\ref{nec}, must be $\lmu=2$ and we get a contradiction to our initial assumption, since $ Z(\lmf,\lmu)=Z(2\lmu+3, \lmu)=2-\lmu=0$.   
  \end{proof}

\begin{lemma}\label{even1}
 If $V$ is non-symmetric and $D_e=1$, then $V=V_o$.
\end{lemma}
\begin{proof}
  Let us assume that 
$D_e=1$, that is $V_e$ is isomorphic to $ 1A$, $\lmd=1$ and, for every $i\in \Z$, $ a_0= a_{2i}$ and $ s_{0,2}=(1-2\bt) a_0$.  Then, the Miyamoto involution $\tau_1$ is the identity on $V$ and the $\bt$-eigenspace for $\ad_{ a_1}$ is trivial. In particular, $ a_{3}= a_{-1}$, and $V_o$ is isomorphic to either $ 1A$, or $3C(2\bt)$, or $2B$.  Then, $ s_{0,3}= s_{0,1}$ and $V$ is  spanned by $ a_{-1},  a_0,  a_1$, and $ s_{0,1}$.
Suppose  $V_o\cong 1A$. Then, $V$ is  spanned  by $ a_0,  a_1$, and $ s_{0,1}$  and $\tau_0$  is the identity on $V$, in particular the $\bt$-eigenspace for $\ad_{ a_0}$ is trivial. This implies that $V$ is an axial algebra of Jordan type,  and hence it is symmetric, a contradiction.

Now suppose $V_o\cong 3C(2\bt)$ or $2B$.  If $Z(\lmu, \lmf)= 0$, then 
$$Z(\lmf, \lmu)=\frac{(4\bt-1)}{2\bt^3}(\lmf-\bt)
$$ and, from the formula for $ a_{-3}$ in Lemma~\ref{a3}, Equation~(\ref{ast5}) and Theorem~\ref{nec}, we get that 
the quadruple $(\lmu, \lmf, 1, \lmdf)$ must be a solution of the system
\begin{equation}\label{1A2X}
\left \{
\begin{array}{l}
Z(\lmu,\lmf)=0\\
(4\bt-1)(\lmf-\bt)^2=\bt^3 \lmdf\\
-\frac{8}{\bt}(\lmu-\lmf)\left ( 2\bt-\lmu-\lmf \right )+\frac{2(2\bt-1)^2}{\bt^2}(\lmf-\bt)-2(1-\lmdf)=0\\
p_2(\lmu, \lmf, 1,\lmdf)=0\\
p_4(\lmu, \lmf, 1,\lmdf)=0.
\end{array}
\right .
\end{equation}
When $\lmdf=0$, the second equation yields that either $\lmf=\bt$ or $\bt=\frac{1}{4}$. In the former case, we obtain $\lm=\lmf=\bt$ from the first equation and the third gives $\lmdf=1$: a contradiction. In the latter case $\bt=\frac{1}{4}$ the system~(\ref{1A2X})
 is equivalent to 
$$
\left \{
\begin{array}{l}
-\frac{1}{32}\lmu+\frac{7}{512}=0\\
16\lmu=0\\
-564\lmu+53=0
\end{array}
\right .
$$
which does not have any solution in any field $\F$.
Finally, when $\lmdf=\bt$, again we get a contradiction since the system~(\ref{1A2X}) does not have any solution in any field $\F$ and for every $\bt$.

 Hence $Z(\lmu, \lmf)\neq 0$. Then, from the formula for $ s_{0,2}$ in Lemma~\ref{a0s2f} we get $ s_{0,1}=(\lmu-\bt) a_0-\frac{\bt}{2}( a_1+ a_{-1})$, whence $V$ is spanned by $a_{-1}, a_0$, and $a_1$ and $ a_0 a_1=\lmu a_0+\frac{\bt}{2}( a_1- a_{-1})$. Now assume $V_o\cong 2B$. Then $a_{-1}a_1=0$ and we see that $\ad_{a_0}$ has eigenvalues $0$, $1$ and $\bt$, while $\ad_{a_1}$ has eigenvalues $0,1$, and $\lambda_1$.  Hence $\lm_1=2\bt$. Then it is straightforward to see that $\ad_{a_0}$ has $0$-eigenvector $4\bt a_0-(a_1+a_{-1})$ and in oder to satisfy the fusion law $\al\star \al=\{0,1\}$ $bt$ must be equal to $\frac{1}{4}$. Then $v:=a_0-\frac{1}{2}(a_1+a_{-1})$ is an $\al$-eigenvector for $\ad_{a_1}$, but $v^2=\frac{1}{2}a_0+\frac{1}{16}(a_1+a_{-1})\not \in \langle a_1, a_{-1}\rangle$. Therefore, the fusion law $\al \star \al=\{0,1\}$ is not satsfied, a contradiction. Therefore, $V_o\cong 3C(2\bt)$ and thus, having the same dimension, $V=V_o$.
\end{proof}

\begin{lemma}\label{V=Vo}
Let $V$ be non-symmetric and suppose that $V=V_o$. Then $\bt=\frac{1}{3}$, 
$V_o\cong 3C(2\bt)$. Moreover, if $a_{-1}, a_1, u$ is a basis of $V_o$ such that $ab=\bt(a+b-c)$ whenever $\{a,b,c\}=\{a_{-1}, a_1, u\}$, then $a_0=\frac{3}{5}(a_{-1}+a_1)-\frac{2}{5}u$.
\end{lemma}
\begin{proof}
Clearly $V=V_o$ has dimension greater than $1$ and so, by Lemmas~\ref{8dim} and~\ref{evenV8(-1/7)}, $V_o$ is isomorphic to one of the following: $2B$, $3C(2\bt)$, $3C(\bt)$,   $3C(-1)^\times$, $3A(2\bt, \bt)$,  $V_3(\bt)$, $4J(2\bt, \bt)$ or its four dimensional quotient, $4Y(2\bt, \bt)$. 

Suppose $V_o\cong 2B$, so that $V=V_o=\langle a_{-1}, a_1\rangle$ and $a_{-1}a_1=0$.  Since $a_0$ is an idempotent distinct from $a_1$ and $a_{-1}$, we get $a_0=a_{-1}+a_1$. Then both $a_{-1}$ and $a_1$ are $1$-eigenvectors for $\ad_{a_0}$, a contradiction.
\medskip

Suppose $V_o\cong 3C(2\bt)$, so that $V=V_o$ has basis $a_{-1}, a_1, u$ where $u^2=u$, and $ab=\bt(a+b-c)$ whenever $\{a,b,c\}=\{a_{-1}, a_1, u\}$. Let $a_0=x_{-1}a_{-1}+x_1a_1+zu$ be the decomposition of $a_0$ with respect to the basis $a_{-1}, a_1, u$. Since $a_0^{\tau_0}=a_0$ and $\tau_{0}$ swaps $a_{-1}$ and $a_1$, we have immediately that $x_{-1}=x_1$. Then we have
$$
0=a_0(a_1-a_{-1})-(a_1-a_{-1})= (2\bt z+x_1-\bt)(a_1-a_{-1})
$$
whence $x_1=\bt-2\bt z$.
Then,
\begin{eqnarray*}
\lefteqn{0=a_0^2-a_0}\\
&=& x_1(2\bt x_1+x_1-1)(a_{-1}+a_1) -(2\bt x_1^2-4\bt x_1z-z^2+z)u.
  \end{eqnarray*} 
  If $x_1=0$, we get $z=\frac{1}{2}$ but also $z^2-z=0$, a contradiction. Hence $x_1\neq 0$ and so 
  $2\bt x_1+x_1-1=0$, whence $2\bt+1\neq 0$ and $x_1=\frac{1}{2\bt+1}$. Then, since $x_1=\bt-2\bt z$, we get $z=\frac{2\bt^2+\bt-1}{2\bt(2\bt+1)}$ and $a_0^2-a_0=\frac{(3\bt-1)(\bt-1)}{4\bt^2}u$. Since $\bt\neq 1$, we get $\bt=\frac{1}{3}$ and $a_0=\frac{3}{5}(a_{-1}+a_1)-\frac{2}{5}u=\mathbbm 1-u$, where $\mathbbm 1=\frac{3}{5}(a_{-1}+a_1+u)$ is the identity of the algebra $3C(\frac{2}{3})$. A straightforward computation shows that $a_0$ is an axis with eigenvalues $1$, $0$ and $\bt$. 
\medskip

Suppose $V_o\cong 3C(\bt)$, so that $V=V_o$ has basis $a_{-1}, a_1, u$ where $u^2=u$, and $ab=\frac{\bt}{2}(a+b-c)$ whenever $\{a,b,c\}=\{a_{-1}, a_1, u\}$. Let $a_0=x_{-1}a_{-1}+x_1a_1+zu$ be the decomposition of $a_0$ with respect to the basis $a_{-1}, a_1, u$. As in the previous case, $x_{-1}=x_1$.
Then we have
$$
0=a_0(a_1-a_{-1})-(a_1-a_{-1})= (\bt z+x_1-\bt)(a_1-a_{-1})
$$
whence $x_1=\bt-\bt z$.
Then,
\begin{eqnarray*}
\lefteqn{0=a_0^2-a_0}\\
&=& x_1(\bt x_1+x_1-1)(a_{-1}+a_1) -(\bt x_1^2-2\bt x_1z-z^2+z)u.
  \end{eqnarray*} 
Now, if $x_1=0$, then $z=1$ and $a_0=u$. This means that $V=\langle \langle a_0, a_1\rangle \rangle$ is symmetric, a contradiction. Hence $x_1\neq 0$, whence $\bt+1\neq 0$, $x_1=\frac{1}{\bt+1}$ and $z=\frac{\bt^2+\bt-1}{\bt(\bt+1)}$. Then we get $a_0^2-a_0=\frac{(2\bt-1)(\bt-1)}{\bt^2}u\neq 0$, a contradiction.
\medskip

Suppose $V_o\cong 3A(2\bt, \bt)$. Then $V=V_o$ has basis $a_{-1}, a_1, a_3, s_{1,2}$. Let 
$$a_0=x_{-1}a_{-1}+x_1a_1+x_3a_3+zs_{1,2}
$$ be the decomposition of $a_0$ with respect to this basis.  As in the previous cases, $x_{-1}=x_1$. Then, 
$$a_2=a_0^{\tau_1}=x_{3}a_{-1}+x_1(a_1+a_3)+zs_{1,2}
$$
 and 
 $$a_{-2}=a_0^{\tau_{-1}}=x_1(a_{-1}+a_3)+x_3a_1+zs_{1,2}.
 $$ 
 Now we compute
\begin{eqnarray*}
\lefteqn{0=a_0(a_{1}-a_{-1})-\bt(a_1-a_{-1})}\\
&=&\frac{1}{(4\bt-1)}\left [(4\bt-1)(x_1+\bt x_3)-\bt^2(7\bt-2)z-\bt(4\bt-1)\right ](a_1-a_{-1})
\end{eqnarray*}
and
\begin{eqnarray*}
\lefteqn{0=a_0^2-a_0-a_2^2+a_2}\\
&=& \frac{(x_1-x_3)}{(4\bt-1)}[(8\bt^2+2\bt-1)x_1-2\bt^2(7\bt-2)z+(4\bt-1)(x_3-1)](a_{-1}-a_3).
\end{eqnarray*}
 Since, by Lemma~\ref{even1}, $V_e\not \cong 1A$, we have $x_3\neq x_1$ and thus, we get 
\begin{equation}\label{3A1}
(4\bt-1)(x_1+\bt x_3)-\bt^2(7\bt-2)z-\bt(4\bt-1)=0
\end{equation}
and 
\begin{equation}\label{3A2}
(8\bt^2+2\bt-1)x_1+(4\bt-1)x_3-2\bt^2(7\bt-2)z-(4\bt-1)=0.
\end{equation}
Subtracting to Equation~(\ref{3A2}) the double of Equation~(\ref{3A1}) we get 
\begin{equation}
(4\bt-1)(2\bt-1)(x_1-x_3+1)=0
\end{equation}
whence (recall $\bt\not \in \{\frac{1}{2}, \frac{1}{4}\}$ in this case) $x_3=x_1+1$. Thus, Equations~(\ref{3A1}) and~(\ref{3A2}) become
\begin{equation}\label{3A11}
(4\bt-1)(\bt+1)x_1-\bt^2(7\bt-2)z=0.
\end{equation}
Suppose $\bt=-1$. Equation~(\ref{3A11}) gives $9z=0$. If $ \F$ has characteristic other than $3$, then $z=0$ and 
$$
0=a_0^2-a_0=-x_1(3x_1+3)(a_{-1}+a_1+a_3)+2x_1( 3x_1+2)s_{1,2} 
$$
whence it follows $x_1=0$. Therefore $a_0=a_3$, and this is a contradiction since $V$ is non-symmetric. If $\F$ has characteristic $3$, then we get 
$$0=a_0^2-a_0=z(a_{-1}+a_1+a_3)+x_1s_{1,2},
$$ whence $z=x_1=0$ and again the contradiction $a_0=a_3$.

Hence $\bt\neq -1$ and from Equation~(\ref{3A11}) we get
$$
x_1=\frac{\bt^2(7\bt-2)}{(4\bt-1)(\bt+1)}z.
$$
In this case we get
\begin{eqnarray*}
\lefteqn{0=4(4\bt-1)^2(\bt+1)^2(a_0^2-a_0)}\\
&=&\bt^2z
  \left [\bt(18\bt^2-\bt-1)(5\bt-1)(8\bt-1)z \right .\\
  &&\:\:\:\:\:\:\:\:\:\left .+4(18\bt^2-8\bt+1)(\bt+1)(4\bt-1)\right ](a_{-1}+a_1+a_3)\\
  &&  2z \left [ \bt(18\bt^2-\bt-1)(5\bt-1)(10\bt^2-1)z \right .\\
 &&\:\:\:\:\:\:\:\:\:\left .+2(12\bt^2+2\bt-1)(\bt+1)(4\bt-1)(3\bt-1)\right ] s_{1,2}
  \end{eqnarray*}
and so either $z=0$, whence $x_1=0$ and $a_0=a_3$, a contradiction as above, or 
\begin{equation}\label{3A3}
\bt(18\bt^2-\bt-1)(5\bt-1)(8\bt-1)z+4(18\bt^2-8\bt+1)(\bt+1)(4\bt-1)=0
\end{equation}
and 
\begin{equation}\label{3A4}
\bt(18\bt^2-\bt-1)(5\bt-1)(10\bt^2-1)z+2(12\bt^2+2\bt-1)(\bt+1)(4\bt-1)(3\bt-1)=0.
\end{equation}
We now multiply Equation~(\ref{3A3}) by $(10\bt^2-1)$ and subtract Equation~(\ref{3A4}) multiplied by $(8\bt-1)$ to get 
$$
2(2\bt-1)^2(18\bt^2-\bt-1)(\bt+1)(4\bt-1)=0.
$$
Therefore, $18\bt^2-\bt-1=0$. Then Equation~(\ref{3A3}) gives $18\bt^2-8\bt+1=0$. Summing and subtracting these two equations we get $9\bt(4\bt-1)=0$ and $7\bt-2=0$, from which it follows that $\F$ has characteristic $3$ and $\bt=-1$, a contradiction to our assumption.
\medskip

Suppose $V_o\cong V_3(\bt)$, with $18\bt^2-\bt-1=0$ and $\F$ of characteristic other than $3$. Then $V=V_o$ has basis $a_{-1}, a_1, a_3$. Let $a_0=x_{-1}a_{-1}+x_1a_1+x_3a_3$ be the decomposition of $a_0$ with respect to this basis.  As in the previous cases,  $x_{-1}=x_1$. Then
\begin{equation*}
0=a_0(a_{1}-a_{-1})-\bt(a_1-a_{-1})=(x_3\bt+x_1-\bt)(a_1-a_{-1})
\end{equation*}
whence $x_1=\bt(1-x_3)$, 
and 
\begin{eqnarray*}
\lefteqn{0=a_0^2-a_0}\\
&=&(x_3-1)\bt[3\bt(\bt-1)x_3-3\bt^2-\bt+1](a_{-1}+a_1)\\
&&+(x_3-1)[(\bt^3-6\bt^2+1)x_3-\bt^3]a_3.
  \end{eqnarray*}
If $x_3=1$, then $x_1=0$ and $a_0=a_3$, a contradiction since $V$ is not symmetric. Hence $x_3\neq 1$ and so $3\bt(\bt-1)x_3-3\bt^2-\bt+1=0$ and $(\bt^3-6\bt^2+1)x_3-\bt^3=0$.  Since $\bt\neq 1$, from the first equation we see that $\F$ has characteristic other than $3$ and $x_3=\frac{3\bt^2+\bt-1}{3\bt(\bt-1)}$, $x_1=\frac{1-4\bt}{3(\bt-1)}$. Then, from the second equation we get 
$$
(7\bt^2-1)(2\bt-1)(\bt+1)=0.
$$
Suppose $7\bt^2-1=0$. Then, since by hypothesis $18\bt^2-\bt-1=0$, we get that $\F$ has caracteristic $19$, $\bt=7$ and $a_0=8(a_{-1}+a_1+a_3)$. It is straightforward so check that the $\bt$-eigenspace for $\ad_{a_0}$ is spanned by $v_1:=3(a_{-1}+a_1)+13a_3$ and $v_2:=3(a_{-1}+a_3)+13a_1$ and since $v_1v_2=6a_{-1}-4(a_1+a_3)\not \in \langle a_0\rangle$, the fusion law $\bt\star \bt=\{0,1,\al\}$ is not satisfied. 
Therefore $\bt=-1$. In this case $a_0=-\frac{5}{6}(a_{-1}+a_{1})+\frac{1}{6}a_3$. It is straightforward to check that $\ad_{a_0}$ has characteristic polynomial $(x-1)(x+1)(x-3)$. Since $\F$ has characteristic other than $3$, must be $3=2\bt=-2$, that is $\F$ has characteristic $5$. Then $a_0=a_3$, a contradiction.
\medskip

Suppose $V_o\cong 3C(-1)^\times$ (and $\F$ has characteristic other than $3$ since $\al=2\bt\neq 1$). Then $V=V_o$ has basis $a_{-1}, a_1$ and $a_{-1}a_1=-a_{-1}-a_1$. Let $a_0=x_{-1}a_{-1}+x_1a_1$ be the decomposition of $a_0$ with respect to this basis.  We have
\begin{equation*}
0=a_0(a_1-a_{-1})-\bt(a_1-a_{-1})=-(2x_{-1}-x_1+1)a_{-1}-(x_{-1}-2x_1-1)a_1
\end{equation*}
whence $x_{-1}=x_1=-1$. Thus $a_0=-a_{-1}-a_1$. Then $\tau_{-1}$ swaps $a_1$ and $a_0$ and hence $V$ is symmetric, a contradiction.
\medskip

Suppose $V_o\cong 4J(2\bt, \bt)$. Then a basis for $V$ is given by 
$a_{-3}, a_{-1}, a_1, a_3, u$, where $u:=-\frac{2}{\bt}s_{1,2}$. Let $a_0=x_{-3}a_{-3}+x_{-1}a_{-1}+x_1a_1+x_3a_3+zu$ be the decomposition of $a_0$ with respect to this basis. As in the previous cases, since $a_0^{\tau_0}=a_0$, we have $x_{-3}=x_3$ and $x_{-1}=x_1$. Then, using the multiplication for the chosen basis given in~\cite[Table~7]{DM}, we get
\begin{eqnarray*}
\lefteqn{0=a_0(a_1-a_{-1})-\bt(a_1-a_{-1})=}\\
&=&  \bt(x_{3}-z)(a_3-a_{-3})+(x_{1}+\bt x_{3}+3\bt z-\bt)(a_1-a_{-1})
\end{eqnarray*}
and 
\begin{eqnarray*}
\lefteqn{0=a_0(a_3-a_{-3})-\bt(a_3-a_{-3})=}\\
&=& (\bt x_1+3\bt z+x_3-\bt)(a_3-a_{-3})+\bt(x_{1}-z)(a_{1}-a_{-1}) 
\end{eqnarray*}
whence, $x_1=x_3=z$ and $(4\bt+1) z-\bt=0$. Thus, if $\bt=-\frac{1}{4}$ we get  immediately $\bt=0$, a contradiction. Hence $\bt\neq -\frac{1}{4}$, $z=\frac{\bt}{4\bt+1}$ and $a_0=\frac{\bt}{4\bt+1}(a_{-3}+a_{-1}+a_1+a_3+u)$. Then, $a_0^2=\bt a_0\neq a_0$, a contradiction.

Suppose now that $\bt=-\frac{1}{4}$ and $V=V_o$ is isomorphic to the four dimensional quotient of $4J(2\bt, \bt)$, $4J(2\bt, \bt)/I$ where $I=\langle a_{-3}+a_{-1}+a_1+a_3+u\rangle$. Then the calculations used above for the case of $4J(2\bt, \bt)$ are easily adopted and produce a contradiction. 
\medskip

Assume finally that $V=V_o\cong Y_4(2\bt, \bt)$, with  $4\bt^2+2\bt-1=0$. Then a basis for $V$ is given by 
$a_{-3}, a_{-1}, a_1, a_3, s$. Let $a_0=x_{-3}a_{-3}+x_{-1}a_{-1}+x_1a_1+x_3a_3+zs_{1,2}$ be the decomposition of $a_0$ with respect to this basis. As above, $x_{-3}=x_3$ and $x_{-1}=x_1$. Then, using the multiplication for the chosen basis given in~Table~3, we get
\begin{eqnarray*}
\lefteqn{0=a_0(a_1-a_{-1})-\bt(a_1-a_{-1})=}\\
&=& -[ \bt x_1+\bt x_3+\frac{1}{2}x_3-\frac{1}{2}\bt^2 z-\bt]a_{-3}+ [ \bt x_{1}+\bt x_{3}+\frac{1}{2}x_1-\frac{1}{2}\bt^2 z-\bt]a_{-1}\\
&&+[(\bt (x_1-x_3)+\frac{1}{2}x_3   +\frac{1}{2}\bt^2 z]a_1+[ (\bt(x_1-x_3)-\frac{1}{2}x_1-\frac{1}{2}\bt^2 z]a_3\\
  &&-4\bt (x_1-x_3 )s
\end{eqnarray*}
whence $x_1=x_3$,
and 
\begin{eqnarray*}
\lefteqn{0=2a_0(a_3-a_{-3})-\bt(a_3-a_{-3})=}\\
&=&(x_1+\bt^2 z)(a_{-1}-a_{-3})+\left [ (4\bt+1 )x_{1}-\bt^2 z-2\bt \right ](a_1-a_3)
  \end{eqnarray*}
  whence $x_1=-\bt^2 z$ and $\bt(2\bt+1) z+1=0$. Thus, if $\bt=-\frac{1}{2}$ we get  immediately a contradiction. If $\bt\neq -\frac{1}{2}$, then $z=-\frac{1}{\bt(2\bt+1)}$ and 
  $$
  a_0=\frac{\bt}{2\bt+1}(a_{-3}+a_{-1}+a_1+a_3)-\frac{1}{\bt(2\bt+1)}s.
  $$ From the condition $a_0^2=a_0$ we get that 
 $$
 32\bt^5+16\bt^4-16\bt^3+4\bt^2+5\bt-1=0.
 $$
Since also $4\bt^2+2\bt-1=0$, this leads to a contradiction. 
\end{proof}

{\it Proof of Theorem~\ref{thm2}.}
Let $V$ be a non-symmetric $2$-generated primitive axial algebra of Monster type $(2\bt, \bt)$ with generators $a_0$ and $a_1$.  If $V=V_o$ or $V=V_e$, then Lemma~\ref{V=Vo} yields that claim (3) holds. So, let us assume that  $V_e\neq V\neq V_o$. By Lemmas~\ref{8dim}, \ref{evenV8(-1/7)}, \ref{evenV5}, \ref{even33}, and ~\ref{even1} the even and the odd subalgebras $V_e$ and $V_o$ are isomorphic to either $2B$ or $3C(2\bt)$ and by Lemma~\ref{simple}, $(\lmd, \lmdf)\in \{(0,0), (\bt, \bt), (0,\bt), (\bt, 0)\}$.  Moreover, $a_{-1}=a_3$, $a_1=a_{-3}$, $a_{4}=a_0$, $a_{-2}=a_2$, and from Lemma~\ref{a3} we get that $(\lmu, \lmf, \lmd, \lmdf)$ satisfies  the following equations
\begin{equation}\label{AA}
Z(\lmu, \lmf)(\lmu-\bt)=\frac{\lmd}{2}
\end{equation}
and 
\begin{equation}\label{BB}
Z(\lmf, \lmu)(\lmf-\bt)=\frac{\lmdf}{2}.
\end{equation}
Suppose first that $\lmd=\lmdf$. Then we get $Z(\lmu, \lmf)(\lmu-\bt)=Z(\lmf, \lmu)(\lmf-\bt)$, which is equivalent to 
$(\lmu-\lmf)(\lmu+\lmf-2\bt)=0$ and so $\lmu+\lmf-2\bt=0$, since $\lmu\neq \lmf$ as the algebra is non-symmetric.
Then, Equations~(\ref{AA}) and~(\ref{BB})  are equivalent to 
\begin{equation}\label{AAA}
\frac{1}{\bt^2}(\lmu-\bt)^2=\frac{\lmd}{2} \:\:\mbox{ and }\:\: \frac{1}{\bt^2}(\lmf-\bt)^2=\frac{\lmdf}{2} 
\end{equation}
If $\lmd=\lmdf=0$, we get the solution $(\bt, \bt, 0,0)$ which corresponds to a symmetric algebra, a contradiction.
Suppose $\lmd=\lmdf=\bt$.  Then it is long but straightforward to check that there is no quadruple $(\lmu, \lmf, \bt, \bt)$ which is a common solution of~(\ref{AAA}) and of the set of polynomials $T$ defined in (\ref{T}). A contradiction to Theorem~\ref{nec}.

Finally assume that $\lmd=\bt$ and $\lmdf=0$. Then, by Equation~(\ref{BB}), either $Z(\lmf, \lmu)=0$ or $\lmf=\bt$. 
If $Z(\lmf, \lmu)=0$, we check that no quadruple $(\lmu, \lmf, \bt, 0)$ is a common solution of Equation~(\ref{AA}) and of the set of polynomials $T$ defined in (\ref{T}). So $\lmf=\bt$. Then Equation~(\ref{AA}) becomes
$$
(\lmu-\bt)^2=\frac{\bt^2}{4}
$$
and we get the two quadruples $(\frac{3}{2}\bt, \bt, \bt, 0)$ and $(\frac{\bt}{2}, \bt, \bt, 0)$. A direct check shows that the former one is not a solution of the set $T$, contradicting Theorem~\ref{nec}. If $(\lmu, \lmf, \lmd, \lmdf)=(\frac{\bt}{2}, \bt, \bt, 0)$, then from Lemma~\ref{a0s2f} we get $s_{0,1}=-\bt( a_0+ a_2)$ and so $V$ has dimension at most $4$. Moreover, $V$ satisfies the same multiplication table as the algebra  $Q_2(\bt)$. By Theorem~8.6 in~\cite{DM}, for $\bt\neq -\frac{1}{2}$ the algebra $Q_2(\bt)$ is simple, while it has a $3$-dimensional quotient over the radical $\F( a_0+ a_1+ a_2+ a_{-1})$ when $\bt=-\frac{1}{2}$. Thus claim (2) follows. \hfill $\square$



\begin{thebibliography}{99}




\bibitem{Singular}
Decker, W.; Greuel, G.-M.; Pfister, G.; Sch{\"o}nemann, H.: 
\newblock {\sc Singular} {4-1-1} --- {A} computer algebra system for polynomial computations.
\newblock {http://www.singular.uni-kl.de} (2018).

\bibitem{FMI}  Franchi, C., Mainardis, M., Ivanov, A.A., Majorana representations of the finite groups.  \textit{Alg. Coll.} {\bf 27} (2020) 
 31-50.
 
 \bibitem{FM}  Franchi, C., Mainardis M., Classifying $2$-generated symmetric axial algebras of Monster type, \textit{J. Algebra} {\bf 596} (2022), 200--218.
 
 \bibitem{FMM} Franchi, C., Mainardis M., M\textsuperscript{c}Inroy, J., Quotients of the Highwater algebra and its cover, \textit{arXiv}:2205.02200,  46 pages, May 2022.


\bibitem{FMS1} Franchi, C., Mainardis, M., Shpectorov, S., $2$-generated axial algebras of Monster type, 
\textit{arXiv}:2101.10315,   22 pages, Jan 2021.


\bibitem{HW} Franchi, C., Mainardis, M., Shpectorov, S., An infinite-dimensional $2$-generated primitive axial algebra of Monster type, \textit{Ann. Mat. Pura Appl.} {\bf 201} (2022), 1279-1293. 


\bibitem{DM} Galt, A., Joshi, V., Mamontov, A., Shpectorov, S., Staroletov, A., Double axes and subalgebras of Monster type in Matsuo algebras, \textit{Comm. Algebra} {\bf 49}, vol. 10, 4208--4248.

\bibitem{GAP} The GAP Group, GAP -- Groups, Algorithms, and Programming, Version 4.10.0; 2019. (https://www.gap-system.org)


\bibitem{HRS} Hall, J., Rehren, F., Shpectorov, S.: Universal Axial Algebras and a Theorem of Sakuma, {\it J.~Algebra} {\bf 421} (2015), 394-424.

\bibitem{HRS2} Hall, J., Rehren, F., Shpectorov, S.: Primitive axial algebras of Jordan type, {\it J.~Algebra} {\bf 437} (2015), 79-115.

\bibitem{JoshiPhD} V. Joshi, \textit{Axial algebras of Monster type $(2\eta, \eta)$}, PhD thesis, University of Birmingham, 2020.

\bibitem{KMS} Khasraw, S.M.S., McInroy, J., Shpectorov, S.: On the structure of axial algebras,  {\it Trans. Amer. Math. Soc.}  {\bf  373} (2020), 2135-2156.

\bibitem{IvInd}  Ivanov, A. A., The Future of Majorana Theory, in: {\it Group Theory and Computation},  {Indian Statistical Institute Series},  Springer, Singapore, 2018, 107-118.  
 
\bibitem{IPSS10} Ivanov,  A. A., Pasechnik, D. V., Seress, \'A., Shpectorov, S.: Majorana representations of the symmetric group of degree $4$,  {\it J. Algebra} {\bf  324} (2010), 2432-2463 

\bibitem{forbidden} M\textsuperscript{c}Inroy, J. and Shpectorov, S., From forbidden configurations to a classification of some axial algebras of Monster type, \textit{arXiv}:2107.07415, 41 pages, Jul 2021.

\bibitem{N96} Norton, S. P.: The Monster algebra: some new formulae. In Moonshine, the Monster and related topics (South Hadley, Ma., 1994), Contemp. Math. 193, pp. 297-306. AMS, Providence, RI (1996)

\bibitem{RT} Rehren, F., Axial algebras, PhD thesis, University of Birmingham, 2015.

\bibitem{R} Rehren, F., Generalized dihedral subalgebras from the Monster, {\it Trans. Amer. Math. Soc.} {\bf 369} (2017), 6953-6986.

\bibitem{S} Sakuma, S.: $6$-transposition property of $\tau$-involutions of vertex operator algebras. Int. Math. Res. Not. (2007). doi:10.1093/imrn/rmn030

\bibitem{stout} Stout M.: Modular representation theory and applications to decomposition algebras, Master thesis, Gent University 2021,  https://algebra.ugent.be/ $\tilde{}$ tdemedts/research-students.php\# 

\bibitem{Turner} Turner, M.: Private communication, Sept 2022.

\bibitem{Yabe} Yabe, T.: On the classification of $2$-generated axial algebras of Majorana type, \textit{arXiv}:2008.01871, 34 pages, Aug 2020.


\end{thebibliography}
\end{document}